\documentclass{amsart}
\usepackage[utf8]{inputenc}
\usepackage{amsmath,amsthm}
\usepackage{txfonts}
\usepackage{mathtools}
\usepackage{tensor}
\usepackage{comment}
\usepackage{amssymb}
\usepackage{stmaryrd}
\usepackage{xpatch}
\usepackage{subfiles}
\usepackage[hidelinks]{hyperref}
\usepackage{color}
\usepackage[left=3cm,right=3cm,top=3cm,bottom=4.5cm]{geometry}
\usepackage{pdfsync}
\usepackage{tikz-cd}
\usepackage{url}
\usepackage{tikz}
\usetikzlibrary{decorations.markings}
\usepackage{enumitem}
\hypersetup{linktoc=all}
\usepackage[capitalise,noabbrev]{cleveref}

\xpatchcmd{\paragraph}{\normalfont}{{\normalfont\bfseries}}{}{}

\newcommand{\defterm}[1]{\emph{#1}}

\numberwithin{equation}{section}

\theoremstyle{plain}
	\newtheorem{thm}[equation]{Theorem}
	\newtheorem*{thm*}{Theorem}
 	\newtheorem*{thma*}{Theorem A}
	\newtheorem*{thmb*}{Theorem B}
	\newtheorem*{thmc*}{Theorem C}
	\newtheorem{cor}[equation]{Corollary}
	\newtheorem*{cor*}{Corollary}
	\newtheorem{prop}[equation]{Proposition}
	\newtheorem*{prop*}{Proposition}
	\newtheorem{lem}[equation]{Lemma}
	\newtheorem*{lem*}{Lemma}
	
	\newtheorem*{sublem*}{Sub-Lemma}
	
	\newtheorem*{ex*}{Exercise}
	
	\newtheorem*{claim*}{Claim}
	
	\newtheorem*{question*}{Question}
	\newtheorem{fact}[equation]{Fact}
	\newtheorem*{fact*}{Fact}
\theoremstyle{definition}
	\newtheorem{Def}[equation]{Definition}
	\newtheorem*{Def*}{Definition}
	\newtheorem{obs}[equation]{Observation}
	\newtheorem*{obs*}{Observation}
	\newtheorem{rmk}[equation]{Remark}
	\newtheorem*{rmk*}{Remark}
	
	\newtheorem{soln*}{Solution}
	
	\newtheorem*{note*}{Note}
	\newtheorem{eg}[equation]{Example}
	\newtheorem*{eg*}{Example}	
	
	\newtheorem*{construction*}{Construction}
	
	\newtheorem*{warning*}{Warning}
	
	\newtheorem*{conj*}{Conjecture}
	\newtheorem{notation}[equation]{Notation}
	\newtheorem*{notation*}{Notation}	
	
	\newtheorem*{recall*}{Recall}
    \newtheoremstyle{component}{}{}{}{}{\itshape}{.}{.5em}{\thmnote{#3}#1}
    \theoremstyle{component}

\newcommand{\nats}{\mathbb{N}}

\DeclareFontFamily{U}{min}{}
\DeclareFontShape{U}{min}{m}{n}{<-> udmj30}{}
\newcommand\yo{\!\text{\usefont{U}{min}{m}{n}\symbol{'207}}\!}

\newcommand{\id}{\mathrm{id}}

\newcommand{\Hom}{\mathrm{Hom}}

\newcommand{\Ob}{\operatorname{Ob}}

\newcommand{\inv}{^{-1}}

\newcommand{\Psh}{\mathsf{Psh}}

\DeclareMathOperator{\Nat}{Nat}

\newcommand{\Set}{\mathsf{Set}}

\newcommand{\Cat}{\mathsf{Cat}}

\newcommand{\sSet}{\mathsf{sSet}}

\newcommand{\calB}{\mathcal{B}}
\newcommand{\calC}{\mathcal{C}}
\newcommand{\calD}{\mathcal{D}}

\newcommand{\calS}{\mathcal{S}}

\newcommand{\Spaces}{\mathsf{Spaces}}

\newcommand{\Gaunt}{\mathsf{Gaunt}}

\newcommand{\sCat}{\mathsf{sCat}}

\newcommand{\globe}{\mathbb{G}}

\DeclareMathOperator{\sd}{sd}
\newcommand{\nd}{\mathrm{n.d.}}

\DeclareMathOperator{\Pow}{Pow}
\newcommand{\fin}{\mathrm{fin}}

\newcommand{\Catinc}{\mathsf{Cat}_N^{\mathrm{f}}}

\newcommand{\nerve}{\nu}

\makeatletter
\newcommand{\genericinjlim}[1]{%
  \mathop{\mathpalette\varlim@{{#1}{\rightarrowfill@\textstyle}}}\nmlimits@
}
\def\varlim@#1#2{\varlim@@#1#2}
\def\varlim@@#1#2#3{%
  \vtop{\m@th\ialign{##\cr
    \hfil$#1\operator@font #2$\hfil\cr
    \noalign{\nointerlineskip\kern1.5\ex@}#3\cr
    \noalign{\nointerlineskip\kern-\ex@}\cr}}%
}
\makeatother

\makeatletter
\newcommand{\genericprojlim}[1]{%
  \mathop{\mathpalette\varlim@{{#1}{\leftarrowfill@\textstyle}}}\nmlimits@
}
\def\varlim@#1#2{\varlim@@#1#2}
\def\varlim@@#1#2#3{%
  \vtop{\m@th\ialign{##\cr
    \hfil$#1\operator@font #2$\hfil\cr
    \noalign{\nointerlineskip\kern1.5\ex@}#3\cr
    \noalign{\nointerlineskip\kern-\ex@}\cr}}%
}
\makeatother

\renewcommand{\varinjlim}{\genericinjlim{lim}}

\renewcommand{\varprojlim}{\genericprojlim{lim}}

\makeatletter
\let\tikz@lib@matrix@start@cell=\tikz@lib@matrix@normal@start@cell
\makeatother

\newcommand\ringring[1]{%
  {
   \mathop{\kern0pt #1}\limits^{
     \vbox to-1.85ex{
       \kern-2ex 
       \hbox to 0pt{\hss\normalfont\kern.1em \r{}\kern-.45em \r{}\hss}%
       \vss 
     }
   }
  }
}

\newcommand{\inc}{\mathrm{f}}

\newcommand{\Free}{F}

\DeclareMathOperator{\colim}{colim}

\newcommand{\Decomp}{\operatorname{Sd}}
\newcommand{\Decompat}[1]{\operatorname{Sd}_{#1}}
\newcommand{\myPoset}{\operatorname{At}}
\newcommand{\Decompbot}{\operatorname{Sd}^\triangleleft}
\newcommand{\Decompatbot}[1]{\operatorname{Sd}^\triangleleft_{#1}}

\newcommand{\lin}{\operatorname{Lin}}
\newcommand{\linbot}{\operatorname{Lin}^\triangleleft}
\newcommand{\DC}{\operatorname{DC}}
\newcommand{\DCbot}{\operatorname{DC}^\triangleleft}

\newcommand{\Lan}{\mathsf{Lan}}

\title{An $(\infty,n)$-Categorical Pasting Theorem}
\author{Timothy Campion}
\date{October 2023}

\begin{document}

\begin{abstract}
    We identify a reasonably large class of pushouts of strict $n$-categories which are preserved by the ``inclusion" functor from strict $n$-categories to weak $(\infty,n)$-categories. These include the pushouts used to assemble from its generating cells any object of Joyal's category $\Theta$, any of Street's orientals, any lax Gray cube, and more generally any ``regular directed CW complex." More precisely, the theorem applies to any \emph{torsion-free complex} in the sense of Forest -- a corrected version of Street's \emph{parity complexes}.

    This result may be regarded as partial progress toward Henry's conjecture that the pushouts assembling any non-unital computad are similarly preserved by the ``inclusion" into weak $(\infty,n)$-categories. In future work we shall apply this result to give new models of $(\infty,n)$-categories as presheaves on torsion-free complexes, and to construct the Gray tensor product of weak $(\infty,n)$-categories.

    This result is deduced from an \emph{$(\infty,n)$-categorical pasting theorem}, in the spirit of Power's 2-categorical and $n$-categorical pasting theorems, and the $(\infty,2)$-categorical pasting theorems of Columbus and of Hackney, Ozornova, Riehl, and Rovelli. This says that, when assembling a ``pasting diagram" from its generating cells, the space of ``composite cells which can be pasted together from all of the generators" is contractible.
\end{abstract}

\maketitle

\tableofcontents


\section{Introduction}

Higher category theory is littered with different formalisms giving models for computing certain colimits of \emph{strict} $n$-categories of a ``combinatorial" flavor.\footnote{Among these formalisms, we mention \emph{computads} (introduced in dimension $2$ by \cite{street-computad} and in general under the name \emph{polygraph} by \cite{burroni}), \emph{parity complexes} \cite{street}, \emph{pasting schemes} \cite{johnson}, \emph{augmented directed complexes} \cite{steiner}, \emph{pasting schemes} \cite{power2} \cite{powern},
\emph{pre-polytopes with labeled structures} \cite{nguyen}, \emph{polyplexes} \cite{henry-regular}, \emph{regular polygraphs} \cite{hadzihasanovic}, \emph{torsion-free complexes} \cite{forest}, \emph{framed topological spaces} \cite{dorn-douglas}, etc. The combinatorics of the strict $n$-categories which result is fearsome in general, and the proliferation of such theoretical frameworks reflects the enormity of the general problem of understanding these combinatorics.} The goal of this note is to leverage the achievements of these category theorists and apply their work directly to the study of weak $(\infty,n)$-categories. We have three closely related results, each of which has antecedents:

\begin{enumerate}
    \item[(A)] In works such as \cite{johnson}, \cite{power2}, and \cite{powern}, certain classes of ``pasting diagrams" are considered. It is shown that each ``pasting diagram" has a unique ``composite". We generalize such results from the strict to the weak setting in \cref{thm:sc-main} below.
    \item[(B)] In works such as \cite{street} (corrected in \cite{forest}) and \cite{steiner}, certain particularly nice computads\footnote{Recall that the \defterm{computads} \cite{street-computad}, or \defterm{polygraphs} \cite{burroni} the class of strict $\omega$-categories which can be built (from the empty category) by iteratively gluing on one cell $\globe_n$ at a time, along its boundary $\partial \globe_n$. Compare to the class of CW complexes -- those topological spaces which can be built (from the empty space) by gluing on one disk $D^d$ at a time along its boundary $S^{d-1}$. The ``particularly nice" computads include \defterm{parity complexes} \cite{street}, or \defterm{torsion-free complexes} \cite{forest}, and can be thought of as directed analogs of the \defterm{regular CW complexes} (those for which the attaching maps are injective), though admittedly in the directed setting the ``regularity" conditions are a bit more involved.} are considered, and explicit descriptions of the free strict $n$-categories thereon are given. We show in \cref{thm:mainthm} below that some of these free strict $n$-categories are also free as weak $(\infty,n)$-categories on the same data. Such results generalize results of type (A), since a ``pasting diagram" is really the same thing as a computad whose geometry is such that we can make sense of the notion of the ``composite" of the diagram. A type (B) description of such a computad should be explicit enough to tell us if the set of composites is a singleton (a type (A) result).
    \item[(C)] We can look for statements of the form ``certain pushouts of strict $n$-categories are preserved by the ``inclusion" $\sCat_n \to \Cat_n$ from strict $n$-categories to weak $(\infty,n)$-categories".\footnote{Note that the ``inclusion" functor $\sCat_n \to \Cat_n$ is not fully faithful. Happily, the inclusion $\sCat_n \to \Cat_n^\inc$ into \emph{flagged} $(\infty,n)$-categories \cite{ayala-francis} \emph{is} fully faithful. Since the localization $\Cat_n^\inc \to \Cat_n$ preserves colimits, we are able throughout the paper to work in the flagged setting, so the scare quotes need not concern us.} For example, the pushout $[2] = [1] \ast_{[0]} [1]$ is preserved by the inclusion $\sCat_1 \to \Cat_1$, essentially by definition; this generalizes to other Segal conditions and to spine inclusions in $\Theta$. See \cite{rezk}, \cite{bsp}, \cite{henry-regular} for examples of such results, with applications to modeling weak $(\infty,n)$-categories. We prove in \cref{cor:po} below a more general pushout-preservation statement which subsumes many of these known results.
\end{enumerate}

In the case $n = 2$, results of type (A) have been previously generalized to $(\infty,2)$-categories in \cite{columbus} and \cite{horr}. Here rather showing that a set of ``composites" of a given pasting diagram, is a singleton, it is shown that a certain space of composites is defined and shown to be contractible. We give a result of type (A) in the weak setting for arbitrary $n$:

\begin{thma*}[see \cref{thm:sc-main}]
    Let $P$ be a torsion-free complex and $\mu$ an cell in $P$. Suppose that $\mu$ is composite. Then the poset $\Decomp(\mu) = \mu \downarrow \Theta_{\nd} \downarrow \Free(\overline{\{\mu\}}) \setminus \{\mu\}$ is contractible.
\end{thma*}

A torsion-free complex \cite{forest} is a particularly nice sort of computad, i.e. a particularly nice sort of strict $n$-category freely built from its cells, about which we shall say more below. The ``poset of decompositions" $\Decomp(\mu) = \mu \downarrow \Theta_{\nd} \downarrow \Free(\overline{\{\mu\}}) \setminus \{\mu\}$ is the poset of ``$\Theta$-shaped decompositions" of $\mu$ into ``smaller" cells. Its elements is the poset of subcategories $\theta$ of the free strict $n$-category $\Free(P)$ on $P$ contained in the ``support" $\Free(\overline{\{\mu\}})$ of $\mu$ which properly contain $\mu$ and are isomorphic to an $n$-category in Joyal's category $\Theta$. See \cref{def:decomp}. In the case $n = 2$, our poset $\Decomp(\mu)$ is not isomorphic to the spaces of composites considered in \cite{columbus} or \cite{horr}, but it plays a similar role.

Our type (B) result (\cref{thm:mainthm}) is formulated by induction on dimension. A computad is a strict $n$-category which can be built by iteratively taking pushouts, gluing on one cell at a time like a ``directed CW complex". Assuming that we understand certain $(d-1)$-computads in the weak setting, the gluing instructions for assmbling the $d$-skeleton of a computad from the $(d-1)$-skeleton by attaching new $d$-cells make just as much sense in the world of weak $(\infty,n)$-categories as they do in strict $n$-categories. So given such assembly instructions, we may assemble all the necessary pushouts either in weak $(\infty,n)$-categories or in strict $n$-categories, and we will have a canonical comparison map from the former construction to the latter. If this comparison map is an equivalence, then we are in business: our ``weakly free $n$-category" in fact has the property of being strict (and we may then repeat the process in the next dimension). We are in fact able to deduce this in certain cases as an application of Theorem A. More precisely, we show:

\begin{thmb*}[see \cref{thm:mainthm}]
    Let $P$ be a torsion-free complex. Then the following pushout in $\sCat_n$:
    \begin{equation*}
        \begin{tikzcd}
            P_d \times \partial \globe_d \ar[r] \ar[d] & P_d \times \globe_d \ar[d] \\
            \Free(P_{\leq d-1}) \ar[r] & \Free(P)
        \end{tikzcd}
    \end{equation*}
    is preserved by the ``inclusion" $\sCat_n \to \Cat_n$ from strict $n$-categories to weak $(\infty,n)$-categories.
\end{thmb*}

\noindent (Here $\Free(P)$ is the (free $n$-category on the) $d$-skeleton of $P$ and $P_d$ is the set of generating $d$-cells of $P$.) Finding that certain ``weak computads" and strict computads coincide, we now profit by knowing that type (B) results in the existing literature are equally valid in the weak setting.

Note that Theorem B is also an instance of a type (C) result: showing that certain pushouts are preserved by the ``inclusion" $\sCat_n \to \Cat_n$. Rearranging pushouts shows that such instances of type (C) results are really interchangeable with more general type (C) results as in \cref{cor:po}:

\begin{thmc*}[see \cref{cor:po}]
    Consider a pushout diagram of strict $n$-categories where the downward maps are folk cofibrations and the rightward maps are monic:
    \begin{equation*}
        \begin{tikzcd}
            A \ar[r,tail] \ar[d,hook] \ar[dr,phantom, "\ulcorner"{description,very near end}] &
            B \ar[d, hook] \\
            C \ar[r,tail] & D
        \end{tikzcd}
    \end{equation*}
    Suppose that $A,B,C,D$ are free on torsion-free complexes. Then this pushout is preserved by the ``inclusion" $\sCat_n \to \Cat_n$.
\end{thmc*}

\subsection*{Torsion-free complexes} In all of this discussion, the ``nice computads" we use are the \defterm{torsion-free complexes}
(in the sense of \cite{forest}). These are a restricted type of computad similar to Street's parity complexes. They include, for example, all objects of Joyal's category $\Theta$, all of Street's orientals, all lax Gray cubes, and more generally (\cite[3.4.4.22]{forest-thesis}) all loop-free augmented directed complexes in the sense of \cite{steiner}. We give an overview of the theory of torsion-free complexes in \cref{sec:tf}. 
We expect that these results will be adaptable to many other formalisms describing restricted classes of computads. For example, an analog of \cref{cor:po} in dimension 2 was proven in \cite{horr}, using Power's formalism of \emph{pasting schemes} (see also \cite{columbus}); our results in dimension 2 should be closely related, but we have not attempted to compare the formalism of \cite{power2} to that of \cite{forest}.

\subsection*{Flagged $(\infty,n)$-categories}
Theorems A,B,C are all in fact proven in the more general setting of \emph{flagged} weak $(\infty,n)$-categories \cite{ayala-francis}. For $n \in \nats \cup \{\omega\}$, the localization $\Cat_n^\inc \to \Cat_n$ preserves colimits, so the results for $(\infty,n)$-categories follow immediately. In the case $n = \omega$, this includes results for $(\infty,\infty)$-categories with either inductive or coinductive equivalences. Note that this distinction between inductive and coinductive equivalences is completely absent in the flagged setting. See \cref{rmk:to-infty} for some further discussion of this point.

\subsection*{Applications} In the forthcoming \cite{campion-gray} we will use the new supply of homotopy pushouts from Theorem C to construct new models of $(\infty,n)$-categories as presheaves on various sites of computads. Related presheaf models have been considered before (cf. \cite{bsp}). When considering such models, there is generally a trade-off to be made: working with a larger site may be convenient, but generally comes at the cost of a less explicit description of the relevant localization. For instance, presheaves on Joyal's category $\Theta_n$ can be localized explicitly at the spine inclusions to get flagged $(\infty,n)$-categories, but $\Theta_n$ is among the smallest sites one might consider. Larger sites, such as the site $\Upsilon_n$ considered in \cite{bsp}, lead to less explicit descriptions of the localization. The homotopy pushouts considered in the present paper will allow to explicitly describe these localizations when working with larger sites than has previously been feasible. This may be seen as progress toward programs such as Henry's program \cite{henry-regular} (related to the Simpson conjecture) which aims to prove a similar result to Theorem B for the class of \emph{nonunital polygraphs} (which properly includes torsion-free complexes). For our purposes the ability to work with a site which includes the Gray cubes will facilitate our approach to the Gray tensor product in \cite{campion-gray}.

\subsection*{Outline of Proofs} 
The proof of Theorem A occupies \cref{sec:cs}. Here we use Quillen's Theorem A a few times to reduce to a certain subposet $\Decompat{\{k\}}(\mu)$, leaning into the ability to describe cells as well-formed closed sets since we are working in a torsion-free complex.

The proof of Theorem B occupies \cref{sec:ws}, and proceeds as follows. It is harmless to assume that $P = P_{\leq d}$ is $d$-dimensional. We work in the model of $\Theta_n$-spaces. In fact, we can do a bit better and obtain these pushout results before Rezk completion -- i.e. we are working in the $\Theta_n$-space model for the $\infty$-category $\Cat_n^\inc$ of \emph{flagged} $(\infty,n)$-categories (\cite{ayala-francis}, see \cref{sec:weak} for more background). 
We start by writing down a very non-fibrant model $\Free_{d-1} P$ in $\Psh(\Theta)$ for the desired pushout in $\Cat_n^\inc$, given by taking the $(d-1)$-skeleton $\Free (P_{\leq d-1})$ and gluing on the $d$-cells in the category of presheaves on $\Theta_n$. We wish to show that the canonical map $\Free_{d-1} P \to \Free P$ is carried to an equivalence by the localization $L_{\Cat_\omega^\inc}$.

First we prove this when $\Free P \in \Theta_n$. The proof in this case is quite different from the proof of the rest of the theorem and follows from some basic expected properties of wedge sums and suspensions, which we prove in \cref{sec:weak}. These results are similar to those of \cite{ozornova-rovelli-fundamental}, but easier in our case since we work in $\Theta_n$-spaces rather than in complicial sets.

We return now to the main body of the proof that $\Free_{d-1} P \to \Free P$ is an equivalence. Using induction on the number of atoms in $P$, we are free to glue in copies of $\Free P'$ for each proper subcomplex $P' \subset P$, resulting in a more-fibrant model $\Free_\partial P$ for the weak colimit. In fact, if $P$ does not have a unique maximal $d$-cell $\mu$ (a ``big cell") -- not contained in any $\Free P'$ for a proper subcomplex $P' \subset P$ -- then $\Free_\partial P = \Free P$ and we are done. 

Otherwise, we next need to glue in all of the $\Theta$-cells which compose to $\mu$. We do this with a ``pre-fibrant replacement" which we call $\Free_\partial^+ P$. The inclusion $\Free_\partial P \to \Free_\partial^+ P$ is seen to be an equivalence by taking its skeletal filtration and observing that each stage is obtained by gluing in extensions of the form $\theta \times \partial \Delta[k] \cup_{\Free_\partial \theta \times \partial \Delta[k]} \Free_\partial \theta \times \Delta[k] \to \theta \times \Delta[k]$, which are $L_{\Cat_\omega^\inc}$-acyclic by the case of Theorem B where $\Free P \in \Theta$ and the cartesianness of the model structure. The ``collapse map" $\Free_\partial^+ P \to \Free P$ is in fact a levelwise equivalence of presheaves on $\Theta$. This is seen by analyzing its fibers, most of which admit initial objects. The only one which doesn't is the fiber over the big cell $\mu \in \Free P(\globe_n)$. The fiber here is none other than the (classifying space of the) poset $\Decomp(\mu)$ of ``ways to compose $\mu$ from smaller cells", which is contractible by Theorem A.

As indicated above, Theorem C (\cref{cor:po}) follows immediately from Theorem B by cancellation of pushouts.

\subsection*{Outline of Paper} The outline of the paper is as follows. We begin in \cref{sec:tf} with an overview of the theory of torsion-free complexes, and discuss a few properties thereof which we will need later. \cref{sec:weak}, which provides background on weak $(\infty,n)$-category theory, may be read essentially independently of the rest of the paper. After reviewing some facts about $\Theta$-spaces, we prove (\cref{subsec:susp}) that the inclusion $\sCat_\omega \to \Cat_\omega^\inc$, from strict $\omega$-categoriess to flagged weak $(\infty,\infty)$-categories, preserves suspension, and also (\cref{subsec:wedge}) that it preserves certain wedge sums. These results may be compared to \cite{ozornova-rovelli-fundamental}, except that our results are much easier because the $\Theta_N$-space model is tailor-made for such compatibility. In \cref{sec:misc} we collect those observations going into the proof of Theorem A (\cref{thm:sc-main}) which are independent of the theory of torsion-free complexes. \cref{sec:cs} introduces the ``space of composites" $\Decomp(\mu)$ of a cell $\mu$ in a torsion-free complex $P$, and shows that it is contractible (Theorem A / \cref{thm:sc-main}). Finally, in \cref{sec:ws} we deduce from this Theorems B and C (\cref{thm:mainthm} and \cref{cor:po}), showing that the gluing pushouts defining the free category on a torsion-free complexes are preserved by the inclusion $\sCat_\omega \to \Cat_\omega^\inc$ (and thence by the composite $\sCat_\omega \to \Cat_\omega^\inc \to \Cat_\omega$ from strict $\omega$-categories to weak $(\infty,\infty)$-categories).

\begin{rmk*}
    In this introduction, we have generally used a lowercase $n$ for the category number and $d$ for the dimension of cells under consideration. In the main body of the paper, we change this convention: we will generally use a capital $N$ for the category number and a lowercase $n \leq N$ for the dimension of cells under consideration. The category number $N$ is generally immaterial to the discussion, as the inclusions $\Cat_N \to \Cat_M$, $\sCat_N \to \sCat_M$, $\Cat_N^\inc \to \Cat_M^\inc$ are all fully faithful and closed under limits and colimits for $N \leq M \leq \infty$. That is, all colimits may be computed at the smallest value of $N$ for which they make sense, which is always finite in this paper. Alternatively, one may compute all colimits at $N = \omega$ and forget about the category level altogether. We would advocate for the latter perspective, pointing out that although there are some subtleties lurking in the definition of $\Cat_\omega$ having to do with inductive versus coinductive equivalences, these subtleties are absent in the flagged setting. See \cref{rmk:to-infty} for some further discussion of this point.
\end{rmk*}

\subsection{Notation and Conventions}
In this introduction, we have generally used a lowercase $n$ for the category number and $d$ for the dimension of cells under consideration.
In the main body of the paper, we change this convention: we will generally use a capital $N$ for the category number and a lowercase $n \leq N$ for the dimension of cells under consideration.
Let $n \in \nats \cup \{\omega\}$. We write $\sCat_n$ for the strict 1-category of strict $n$-categories. We write $\Cat_n$ for the $(\infty,1)$-category of weak $(\infty,n)$-categories (with inductive equivalences when $n = \omega$), and $\Cat_n^\inc$ for the $(\infty,1)$-category of flagged $(\infty,n)$-categories \cite{ayala-francis}. We sometimes write $A \ast_B C$ for the pushout of $A \leftarrow B \to C$ computed in $\sCat_n$, standing in contrast to $A \cup_B C$ which we usually reserve for a pushout taken in a category of presheaves. A \defterm{space} for us should be taken to mean a simplicial set, and a poset or category will automatically be thought of as a simplicial set by taking its nerve. More notation appears in the body of the paper. In particular \cref{notation:strict} gives not-entirely-standard notation for the boundaries of cells in an $n$-category. We will often say that a poset or category is \defterm{contractible} if its classifying space is contractible. We write $\yo$ for the Yoneda embedding.

\subsection{Acknowledgements}
I would like to thank David Ayala, Gregory Arone, Alexander Campbell, Simon Forest, David Gepner, Philip Hackney, Amar Hadzihasanovic, Sina Hazratpour, Simon Henry, Sam Hopkins, Michael Johnson, Yuki Maehara, David Jaz Myers, Viktoriya Ozornova, Emily Riehl, Martina Rovelli, Maru Sarazola, Chris Schommer-Pries, Benjamin Steinberg, Chaitanya Leena Subramaniam, and Jonathan Weinberger for helpful discussions. I'm grateful for the support of the ARO under MURI Grant W911NF-20-1-0082.

\section{Torsion-free complexes}\label{sec:tf}
In this section we work with strict $\omega$-categories. The purpose of this section is to set up some basic facts surrounding the \emph{torsion-free complexes} of \cite{forest} and \cite{forest-thesis}. These are the most general class of ``parity complex" in the sense of Street (\cite{street}) with good $n$-categorical properties. We also mention some less-standard properties of torsion-free complexes: \cref{prop:theta-reg} shows that the free categories thereon are \defterm{$\Theta$-regular} and \defterm{hypercancellative}. The $\Theta$-regularity property will be a simplifying assumption in the definitions of the poset $\Decomp(\mu)$ (\cref{def:decomp'}) and the construction $\Free^+_\partial P$ (\cref{def:prefib}) below. Hypercancellativity will be useful in \cref{lem:wedge} below (feeding into \cref{lem:theta-main}). We also introduce the preorder $\myPoset_k(\mu)$ on $\mu_k$, which will be key to our understanding of $\Decompat{\{k\}}(\mu)$ in \cref{lem:forest-order} below. Finally, we prove a key well-foundedness property of $\myPoset_k(\mu)$ in \cref{lem:exists-min}, which allows for \cref{cor:forest-sd-ideal} to be applied in the proof of \cref{thm:sc-main} below.

\begin{notation}\label{notation:strict}
    Let $C$ be a strict $\omega$-category, and let $c \in C_n$ be an $n$-cell. 
    
    For $0 \leq i \leq n$, we write $\partial_i c$ for the $i$-dimensional target of $c$, and $\partial_{-i} c$ for the $i$-dimensional source of $c$. In this notation, $0 \neq -0$, and $\partial_n c = c = \partial_{-n} c$.

    If $a,b \in C_n$ are $n$-cells, then we write $b \circ_i a$ for the $i$-dimensional composite of $a,b$, i.e. the composite which requires $\partial_i a = \partial_{-i} b$.
\end{notation}

\begin{Def}[\cite{forest-thesis}]
A \defterm{hypergaph} $P$ comprises sets $(P_n)_{n \in \nats}$ along with functions $(-)^+, (-)^- : P_n \to \Pow_\fin(P_{n-1})$, where $\Pow_\fin(X)$ denotes the set of finite subsets of $X$. The elements of $P_n$ are called the \defterm{$n$-atoms} of $P$. The hypergraph $P$ is \defterm{of dimension $\leq N$} if $P_n = \emptyset$ for $n \geq N+1$.

A \defterm{torsion-free complex} is a hypergraph $P$ subject to certain axioms (T0)-(T4) which may be found in \cite[Section 1.7]{forest}. Among these is the \defterm{acyclicity axiom}, which says that for each $n \in \nats$, the relation $\triangleleft_{n-1}$ on $P_n$, defined by $x \triangleleft_{n-1} y$ if $x^+ \cap y^- \neq \emptyset$, is \emph{acyclic} in the sense that its transitive closure is irreflexive.
\end{Def}

\begin{Def}[\cite{forest-thesis}]\label{def:tf}
Let $P$ be a torsion-free complex. 

For $n \in \nats$, a \defterm{pre-$n$-cell} $c$ in $P$ comprises subsets $(c_{-0}, c_0, c_{-1}, c_1, \dots, c_{-(n-1)}, c_{n-1}, c_{-n}, c_n)$ where $c_i \subseteq C_{|i|}$, $c_{-n} = c_n$. These form a globular set where $\partial_{\epsilon i} c = (c_{-0},c_0, \dots, c_{-(j-1)}, c_{j-1}, c_{\epsilon i}, c_{\epsilon i})$.

For $n \in \nats$ and $S \subseteq P_n$, we write $S^+ = \cup_{s \in S} s^+$, $S^- = \cup_{s \in S} s^-$, $S^\pm = S^+ \setminus S^- \subseteq P_{n-1}$ and $S^\mp = S^- \setminus S^+ \subseteq P_{n-1}$. If $X,Y \subseteq P_{n-1}$, we say that $S$ \defterm{moves} $X$ to $Y$ if $S^\mp = X \setminus Y$ and $S^\pm = Y \setminus X$. (Note that it follows that $Y = (X \setminus S^{\mp}) \cup S^\pm$ and $X = (Y \setminus S^{\pm}) \cup S^{\mp}$.)

For $n \in \nats$ and $S \subseteq P_n$, we say that $S$ is \defterm{fork-free} if either $n = 0$ and $S$ is a singleton, or $n \geq 1$ and for $u \neq v \in S$ we have $u^- \cap v^- = u^+ \cap v^+ = \emptyset$.

A pre $n$-cell $c$ is said to be a \defterm{cell} if for each $c_{-i},c_i$ is fork-free, and for each $0 \leq i \leq n-1$, we have that $c_{i+1}$ and $c_{-(i+1)}$ both move $c_{-i}$ to $c_i$. The cells form a globular subset of pre-cells. To each $n$-atom $a$, there is an associated $n$-cell $\langle a \rangle$, where $\langle a \rangle_{\pm n} = \{a\}$, and $\langle a \rangle_{i-1} = \langle a \rangle_i^\pm$ and $\langle a \rangle_{-(i-1)} = \langle a \rangle_i^\mp$.

If $c,d$ are $n$-cells with $\partial_i c = \partial_{-i} d$, then $d \circ_i c$ is an $n$-cell, defined as follows:
\begin{equation*}
    (d \circ_i c)_{\epsilon j} =
    \begin{cases}
        d_{\epsilon j} = c_{\epsilon j} & j < i \\
        d_i & j = i, \epsilon = + \\
        c_{-i} & j = i, \epsilon = - \\
        c_{\epsilon j} \cup d_{\epsilon j} & j > i
    \end{cases}
\end{equation*}
(The axioms are such that the union appearing above is a disjoint union.) In this way, the cells of $P$ form a strict $\omega$-category which we call $\Free P$.

If $P$ is a torsion-free complex, a \defterm{subcomplex} $Q$ comprises subsets $Q_n \subseteq P_n$ for each $n \in \nats$ which is closed under $(-)^-$ and $(-)^+$. In this case $Q$ itself forms a torsion-free complex. We write $P_{\leq n}$ for the torsion-free complex given by throwing away all atoms in degrees $\geq n+1$. More generally, for any $S \subseteq \cup_n P_n$, there is a smallest subcomplex containing $S$, which we denote $\bar S$. The construction $\Free$ is functorial in subcomplex inclusions.
\end{Def}

\begin{thm}[{\cite[Cor 3.3.3.5]{forest-thesis}}]\label{thm:forest-free}
    Let $P$ be a torsion-free complex and $n \in \nats$. Then for each $p \in P_n$, there is a canonical map $\partial \globe_n \to \Free(P_{\leq n-1})$, such that the square
    \begin{equation*}
        \begin{tikzcd}
            P_n \times \partial \globe_n \ar[r] \ar[d] & P_n \times \globe_n \ar[d] \\
            \Free(P_{\leq n-1}) \ar[r] & \Free(P_{\leq n})
        \end{tikzcd}
    \end{equation*}
    is a pushout in $\sCat_\omega$ (equally, in $\sCat_N$ for $N \geq n$).
\end{thm}

Applying \cref{thm:forest-free} inductively, we see that there is a canonical computad associated to a torsion-free complex $P$, and that $\Free P$ agrees with the free strict $\omega$-category on this computad.

\begin{rmk}\label{rmk:general}
    Torsion-free complexes are quite general, so far as strict $n$-categories of a ``combinatorial" flavor go. For example, they include all loop-free Steiner complexes (\cite[Thm 3.4.4.22]{forest-thesis}). In particular, they include all objects of Joyal's category $\Theta$, all of Street's orientals, and all lax Gray cubes.
\end{rmk}

\begin{Def}
We say that a torsion-free complex $P$ \defterm{has a big cell} if $P$ is itself a finite closed well-formed set of $P$ in the sense of \cite[Section 1.5]{forest}. The finite closed well-formed sets of $P$ correspond bijectively to the cells of $\Free P$ (\cite[Theorem 3.1.21]{forest}), so this means that there is a \defterm{big cell} $\mu: \globe_n \to \Free P$, i.e. a cell which does not factor through $\Free P'$ for any proper subcomplex $P' \subset P$.
\end{Def}

\begin{eg}
If $\theta \in \Theta$, then $\theta = \Free(T)$ for a unique torsion-free complex $T$. Moreover, $T$ has a big cell.
\end{eg}

\begin{Def}\label{def:regularity}
Let $A$ be a strict $\omega$-category.

Say that $A$ is \defterm{finite} if its underlying set is finite.

Say that $A$ is \defterm{$\globe$-regular} if every nondegenerate map $\globe_n \to A$ is monic 
(equivalently, every $\circ_i$-endomorphism in $A$ is of dimension $\leq i$, for any $i \in \nats$).

Say that $A$ is \defterm{$\Theta$-regular} if every nondegenerate map $\theta \to A$ with $\theta \in \Theta$ is monic.

Say that $A$ is \defterm{hypercancellative} if, for $a, b , a', b' \in A$ with $\partial^+_i a = \partial^+_i a' = \partial^-_{i} b = \partial^-_{i} b'$, we have that $(b \circ_i a = b' \circ_i a') \Rightarrow a = a', b= b'$.
\end{Def}

\begin{lem}\label{lem:globe-hyper-theta}
Every $\globe$-regular, hypercancellative strict $\omega$-category is $\Theta$-regular.
\end{lem}
\begin{proof}
Let $\calC \subseteq \Theta$ be the collection of $\theta \in \Theta$ such that for every $\globe$-regular, hypercancellative $A$, every nondegenerate $F : \theta \to A$ is monic. By \cite[Lemma 1.10]{campion-dense}, it will suffice to show that $\ast \in \calC$ (which is trivial) and that $\calC$ is closed under suspension and wedge sums. 

Suppose first that $\theta = \Sigma \zeta$ with $\zeta \in \calC$. Then the induced map $F' : \zeta \to A(F(-),F(+))$ is nondegenerate. For if $F'$ factors through $\zeta \twoheadrightarrow \zeta'$, then $F$ factors through $\Sigma \zeta \twoheadrightarrow \Sigma \zeta'$; by nondegeneracy of $F$, this is an isomorphism, so $\zeta \to \zeta'$ is an isomorphism. Since $A$ is $\globe$-regular and hypercancellative, so is $A(F(-),F(+))$. Since $\zeta \in \calC$, it follows that $F'$ is monic. Therefore, $F$ is monic as soon as we check that $F(-) \neq F(+)$. This holds because otherwise there would be a nontrivial endomorphism in $A$, by $\globe$-regularity. Thus $\theta \in \calC$.

Suppose now that $\theta = \zeta \vee \eta$ with $\zeta,\eta \in \calC$. Then it is again clear that the composite maps $F_0: \zeta \to A$, $F_1 : \eta \to A$ are nondegenerate. As $\zeta,\eta \in \calC$, they are monic. So if $F$ fails to be monic, then we either have an identification of objects from $\zeta,\eta$ other than the wedge point (which contradicts $\globe$-regularity by introducing endomorphisms), or else we have $F(n \circ_0 z) = F(n' \circ_0 z')$ for some cells $z,z' \in \zeta$ and $n,n' \in \eta$, where the composite is over the wedge object. By hypercancellativity of $A$, $F(z) = F(z')$ and $F(n) = F(n')$. Because $F_0,F_1$ are monic, this implies that $z = z'$ and $n = n'$, so that $z \circ_0 n = z' \circ_0 n'$ after all. Thus $\theta \in \calC$.
\end{proof}

\begin{prop}\label{prop:theta-reg}
Let $P$ be a torsion-free complex. Then $\Free P$ is $\Theta$-regular and hypercancellative.
\end{prop}
\begin{proof}
By \cref{lem:globe-hyper-theta}, it suffices to show that $\Free P$ is $\globe$-regular and hypercancellative. $\globe$-regularity follows from Forest's nonemptiness axiom (T0) (\cite[Section 1.7]{forest}). Hypercancellativity is shown in \cite{forest-hyp}.
\end{proof}


\begin{Def}\label{def:pos}
    Let $P$ be a torsion-free complex with an $n$-cell $\mu$. Let $\myPoset_k(\mu)$ be the preorder of elements of $(\partial_k \mu)_k$, whose equivalence classes are generated by declaring two elements $a,b \in (\partial_k \mu)_k$ to be \defterm{equivalent} if they are both in $(\partial_k c)_k$ for some atom $c \in \overline{\{\mu\}}$ of dimension $\geq k+1$, and whose ordering relation is generated by declaring that $a \leq b$ if $a \triangleleft_{k-1} b$.
\end{Def}

    We will be particularly interested in the poset $\lin (\myPoset_k(\mu))$ of non-codiscrete linear preorders refining $\myPoset_k(\mu)$, ordered by reverse containment (\cref{def:linpre}).


\begin{lem}\label{lem:exists-min}
    Let $P$ be a torsion-free complex, and let $\mu$ be an $n$-cell in $P$. If $\mu$ is not an atom, then there exists $1 \leq k \leq n$ such that the preorder $\myPoset_k(\mu)$ is not an equivalence relation. 
\end{lem}
\begin{proof}
    We prove the contrapositive. Suppose that the preorder $\myPoset_k(\mu)$ is an equivalence relation for all $1 \leq k \leq n$. If $\mu_n$ is not a singleton, then the fact that the preorder $\myPoset_n(\mu)$ is a (necessarily discrete) equivalence relation allows us to apply \cite[Lem 3.3.2.2]{forest-thesis} inductively (downward on $k$), we find that there is a $\circ_{-1}$ -decomposition $\mu = b \circ_{-1} a$ where $a$ and $b$ are each of dimension $n$. This is absurd (there is no composition operation $\circ_{-1}$!), so $\mu_n = \{a\}$ is a singleton. Now the equivalence classes of $\myPoset_{n-1}(\mu)$ are all singletons except for the class $a^+$, which generates the cell $\partial_{n-1} \langle a\rangle$. If there are any equivalence classes other than $a^+$, then the lack of nontrivial relations in $\myPoset_{n-1}(\mu)$ again allows us to apply \cite[Lem 3.3.2.2]{forest-thesis} repeatedly to find an absurd $\circ_{-1}$ decomposition of $\partial_{n-1} \mu$. Therefore $\mu_{n-1} = a^+$. Similarly, $\mu_{-(n-1)} = a^-$. Continuing downward in this manner, we find that $\mu = \langle a \rangle$ is an atom.
\end{proof}

\section{Background on weak $\omega$-category theory}\label{sec:weak}
In this section we fix some notation surrounding weak $(\infty,N)$-categories. Along the way, we will verify some routine properties of suspensions (\cref{subsec:susp}) and wedge sums (\cref{subsec:wedge}) using the $\Theta$-space model for $(\infty,N)$-categories \cite{rezk}. These results are similar to the results of \cite{ozornova-rovelli-fundamental}, but they are much easier because we work in the $\Theta$-space model, which very well adapted to the study of suspensions and wedges, rather than the complicial model. We also prove \cref{lem:dist-lat}, a sort of ``locality principle" for $L_{\Cat_N^\inc}$-acyclic maps.

\begin{Def}\label{def:wedge}
Let $\ast \xrightarrow a A$, $\ast \xrightarrow b B$ be pointed objects in an $\infty$-category $\calC$. We write $A \tensor[_a]{\vee}{^\calC_b} B$, or $A \vee^\calC B$ for short, for the coproduct of $(A,a)$ and $(B,b)$ as pointed objects, i.e. the pushout $A \cup_\ast B$ taken with respect to the maps $a$ and $b$.
\end{Def}

\begin{Def}\label{def:susp-strict}
For $n \in \nats \cup \{\omega\}$, let $\Sigma : \sCat_n \to \sCat_{1+n}$ denote the suspension functor. Here $\Sigma C$ is the category with two objects, $0$ and $1$, and $\Hom(0,0) = \Hom(1,1) = [0]$, $\Hom(0,1) = C$, and $\Hom(1,0) = \emptyset$.
\end{Def}

\begin{Def}\label{def:theta}
    We denote by $\Theta$ Joyal's category $\Theta$ -- see e.g. \cite{berger}. This is the smallest full subcategory of $\sCat_\omega$ containing the terminal category $[0]$, closed under suspension, and closed under wedge sums which glue a sink to a source. Then for $N \leq \omega$ we write $\Theta_N = \Theta \cap \sCat_N$. We have also $\Theta = \Theta_\omega = \cup_{N \in \nats} \Theta_N$.
\end{Def}

\begin{Def}\label{def:susp}
Let $N \in \nats \cup \{\omega\}$. Let $\Sigma_/ : \Theta_N \to (\Theta_{1+N})_{/[1]}$ be the suspension functor. Note that $\Sigma_{/}$ admits two canonical lifts $\Sigma^0_/,\Sigma^1_/ : \Theta_N \rightrightarrows (\Theta_{1+N})_{\ast //[1]}$ where $\ast \in \Theta_0$ is the terminal object. Let $\bar \Sigma_/ = \yo_{(\Theta_{1+N})_{/[1]}} \Sigma_/ : \Theta_N \to \Psh((\Theta_{1+N})_{/[1]})$. As the Yoneda embedding preserves the terminal object and presheaves commute with slicing, we have also two lifts $\bar \Sigma^0_/ , \bar \Sigma^1_/ : \Theta_N \rightrightarrows \Psh(\Theta_{1+N})_{\ast//[1]}$. Taking the pullback, we obtain a functor $\bar \Sigma_{//} : \Theta_N \to \Psh(\Theta_{1+N})_{\ast//[1]} \times_{\Psh(\Theta_{1+N})_{/[1]})} \Psh(\Theta_{1+N})_{\ast//[1]} \simeq \Psh(\Theta_{1+N})_{\partial [1]//[1]}$ to bipointed $\Theta_{1+N}$-spaces over $[1]$. Moreover, this functor is fully faithful. We let $\tilde \Sigma_{//} : \Psh(\Theta_N) \to \Psh(\Theta_{1+N})_{\partial [1] //[1]}$ denote the unique colimit-preserving extension (which is again fully faithful), and $\tilde \Hom_{/} : \Psh(\Theta_{1+N})_{\partial [1] //[1]} \to \Psh(\Theta_N)$ its right adjoint. We let $\tilde \Sigma : \Psh(\Theta_N) \to \Psh(\Theta_{1+N})$ denote the composite of $\tilde \Sigma_{//}$ with the forgetful functor.
\end{Def}

\begin{Def}
Let $N \in \nats \cup \{\omega\}$. Let $A$ be a strict $N$-category, and let $a \in A$ be an object. We say that $a$ is a \defterm{source} in $A$ if $\Hom_A(a,b)$ is a point for $a = b$ and empty otherwise. Dually, we say that $a$ is a \defterm{sink} in $A$ if $\Hom_A(b,a)$ is a point for $a = b$ and empty otherwise.
\end{Def}

\begin{eg}
Let $N \in \nats \cup \{\omega\}$, and let $\theta \in \Theta_N$. Then $\theta$ has a unique sink and a unique source. Moreover, if $\theta,\zeta \in \Theta_N$, and if we regard $\theta$ as pointed by its sink and $\zeta$ as pointed by its source, then $\theta \vee^{\sCat_N} \zeta \in \Theta$ as well. When writing wedge sums of objects of $\Theta_N$, we will always assume that the pointings have been chosen in this way.
\end{eg}

\begin{Def}\label{def:spine}
Let $N \in \nats \cup \{\omega\}$.
The set of \defterm{basic wedge inclusions} \cite[Notation 12.1]{bsp} is the following set of morphisms in $\Psh(\Theta_N)$: 
\begin{align*}
    \tilde \Sigma^k(\yo(\theta) \vee^{\Psh(\Theta_N)} \yo(\zeta)) \to \yo(\Sigma^k(\theta \vee^{\sCat_N} \zeta)) \qquad \text{for $\theta, \zeta \in \Theta_{N-k}$}
\end{align*}

Here by convention, $\omega - k = \omega$. The set of \defterm{basic spine inclusions} \cite[Section 5]{rezk} is the following set of morphisms in $\Psh(\Theta_N)$:
\begin{align*}
    \tilde \Sigma^k(\yo(\Sigma \theta_1) \vee^{\Psh(\Theta_N)} \cdots \vee^{\Psh(\Theta_N)} \yo(\Sigma \theta_r)) \to \yo(\Sigma^k(\Sigma \theta_1 \vee^{\sCat_N} \vee \cdots \vee^{\sCat_N} \Sigma \theta_r)) \qquad \text{for $\theta = \Sigma \theta_1 \vee \cdots \vee \Sigma \theta_r \in \Theta_{N-k}$}
\end{align*}
\end{Def}

\begin{lem}[{\cite[Section 13, footnote 2]{bsp}}]\label{lem:same-loc}
Let $N \in \nats \cup \{\omega\}$. The localization of $\Psh(\Theta_N)$ at the basic wedge inclusions coincides with the localization of $\Psh(\Theta_N)$ at the basic spine inclusions.
\end{lem}
\begin{proof}
We factor a basic spine inclusion as 
\begin{align*}
    &\yo(\Sigma \theta_1) \vee^{\Psh(\Theta_N)} \cdots \vee^{\Psh(\Theta_N)} \yo(\Sigma \theta_r) \\
\to& \yo(\Sigma \theta_1 \vee^{\sCat_N} \vee \cdots \vee^{\sCat_N} \Sigma \theta_{s}) \vee^{\Psh(\Theta_N)} \yo(\Sigma \theta_{s+1} \vee^{\sCat_N} \vee \cdots \vee^{\sCat_N} \Sigma \theta_{r}) \\
\to& \yo(\Sigma \theta_1 \vee^{\sCat_N} \vee \cdots \vee^{\sCat_N} \Sigma \theta_r)
\end{align*}
On the one hand, by induction on $r$ and closure under cobase-change, the first map is in the closure of the basic wedge inclusions under cobase-change and composition. Composing with the second map, which is a basic wedge inclusion, it results inductively that the composite map is too. Conversely, the first map is a cobase-change of a basic spine inclusion and so is the composite map, so by two-for-three, the second map (which is a general basic wedge inclusion) is in the closure of the basic spine inclusions under cobase-change and two-for-three. Closing under $\tilde \Sigma$ (which preserves cobase-changes and composition), the statement results.
\end{proof}

\begin{Def}\label{def:theta-space}
Let $N \in \nats \cup \{\omega\}$. We let $\Cat_N^\inc$ denote the localization of $\Psh(\Theta_N)$ at the basic spine inclusions. These are the \defterm{flagged $(\infty,N)$-categories} in the sense of \cite{ayala-francis}. We let $\sCat_N$ denote the localization of $\Cat_N^\inc$ at the maps $\globe_m \times S^d \to \globe_m$ for each $m \leq N$. We let $\Cat_N$ denote the localization of $\Cat_N^\inc$ at the Rezk maps. We let $\Gaunt_N = \sCat_N \cap \Cat_N$ denote the intersection of these two, localized at both classes of maps.

We also denote by $\nerve : \sCat_N \to \Cat_N^\inc \to \Psh(\Theta_N)$ the inclusion functor (this is the ``naive nerve" of strict $N$-categories).
\end{Def}

\begin{rmk}\label{rmk:to-infty}
The notation ``$\sCat_N$" of \cref{def:theta-space} is consistent with the usual definition of strict $N$-categories; the equivalence is induced by the restricted Yoneda embedding. The notation ``$\Gaunt_N$" of \cref{def:theta-space} is consistent with the usage of \cite{bsp}. The notation ``$\Cat_N$" is likewise equivalent with any of the equivalent models considered in \cite{bsp}. When $N= \omega$, we have limit expressions $\Psh(\Theta_\omega) = \varprojlim_{N \in \nats} \Psh(\Theta_N)$, $\Cat_\omega^\inc = \varprojlim_{N \in \nats} \Cat_N^\inc$, $\sCat_\omega = \varprojlim_{N \in \nats} \sCat_N$, $\Cat_\omega = \varprojlim_{N \in \nats} \Cat_N$, and $\Gaunt_\omega = \varprojlim_{N \in \nats} \Gaunt_N$ (cf. \cite[Rmk 2.8]{campion-dense}), where the limits are taken along the forgetful functors.
\end{rmk}

\begin{Def}\label{def:hpo}
Let $N \in \nats \cup \{\omega\}$. We also say that a commutative square of presheaves on $\Theta_N$ is a \defterm{homotopy pushout} if it is preserved by the localization functor $\Psh(\Theta_N) \to \Cat_N^\inc$ (hence also by the localization to $\Cat_N$). We say that a commutative square in $\sCat_N$ is a \defterm{homotopy pushout} if its nerve is a homotopy pushout in $\Psh(\Theta_N)$.
\end{Def}

\begin{rmk}
For $M \leq N \leq \omega$, there is a fully faithful inclusion $\Cat_M^\inc \to \Cat_N^\inc$, which has both left and right adjoints, just as the inclusion $\sCat_M \to \sCat_N$ is fully faithful with both adjoints. In particular, a homotopy pushout of strict $M$-categories is the same thing as a homotopy pushout of strict $N$-categories which happen to be degenerate in dimension $\geq M+1$.
\end{rmk}

\begin{lem}\label{lem:hty-paste}
Let $N \in \nats \cup \{\omega\}$. Consider a commutative diagram as follows in $\sCat_N$ or $\Psh(\Theta_N)$:
\begin{equation*}
\begin{tikzcd}
A \ar[r] \ar[d] & B \ar[r] \ar[d] & C \ar[d] \\
D \ar[r] & E \ar[r] & F
\end{tikzcd}
\end{equation*}
If the first and second square are homotopy pushouts, then so is the composite rectangle. If the first square and the composite rectangle are homotopy pushouts, then so is the second square.
\end{lem}
\begin{proof}
This follows from the fact that pushout squares have the same properties.
\end{proof}

\begin{rmk}\label{rmk:flat}
Let $N \in \nats \cup \{\omega\}$. The $\infty$-categories $\Psh(\Theta_N)$, $\Cat^\inc_N$, $\sCat_N$, $\Cat_N$, and $\Gaunt_N$ may be presented by localizing the injective model structure on simplicial presheaves on $\Theta_N$, as in \cite{rezk}. In these model structures, the cofibrations are the monomorphisms of simplicial presheaves on $\Theta_N$, and every object is cofibrant so that the model structure is left proper. Thus the monomorphisms of simplicial presheaves are \defterm{flat} -- cobase-change along monomorphisms of simplicial presheaves is preserved by the localization functor from simplicial presheaves to $\Cat_N^\inc$, etc. A good supply of monomorphisms of simplicial presheaves (and thus of flat maps) is supplied by the monomorphisms of strict $N$-categories (as the restricted Yoneda embedding into simplicial presheaves preserves monomorphisms).
\end{rmk}

\begin{obs}\label{obs:topos-po}
    Let $N  \in \nats \cup \{\omega\}$. Consider a commutative square of monomorphisms in $\Psh(\Theta_N)$:
    \begin{equation*}
    \begin{tikzcd}
        A \ar[r, tail] \ar[d, tail] & B \ar[d, tail] \\
        C  \ar[r, tail] & D
    \end{tikzcd}
    \end{equation*}
    If the square is a pushout, then  it is  also a pullback. Conversely, if the  square is a pushout, then it is a pullback iff the maps $B \to D \leftarrow C$ are jointly surjective. This observation is true in any 1-topos.
\end{obs}

\begin{lem}\label{lem:flat}
    Let $N  \in  \nats \cup \{\omega\}$. Let $A, B \subseteq X$ be subobjects in $\Psh(\Theta)$. Then the square
    \begin{equation*}
        \begin{tikzcd}
            A \cap  B \ar[r, tail] \ar[d, tail] & A \ar[d, tail] \\
            B \ar[r, tail] & A \cup B
        \end{tikzcd}
    \end{equation*}
    is a pushout preserved by $L_{\Cat_N^\inc}$. Thus, if $A \cap  B \rightarrowtail B$ is an $L_{\Cat^\inc_N}$-equivalence, then $A \rightarrowtail A \cup B$ is likewise an $L_{\Cat^\inc_N}$-equivalence.
\end{lem}
\begin{proof}
    This results by combining \cref{rmk:flat} and \cref{obs:topos-po}.
\end{proof}

\begin{lem}\label{lem:hty-retract}
Let $N \in \nats \cup \{\omega\}$. 
Homotopy pushout squares in $\sCat_N$ are closed under retracts.
\end{lem}
\begin{proof}
This follows from the fact that pushout squares in any $\infty$-category are closed under retracts.
\end{proof}

Before studying homotopy pushouts in more detail, here is a technical lemma which permits an inductive approach to constructing weak equivalences of $\Theta_N$-spaces.

\begin{lem}\label{lem:dist-lat}
Consider a presheaf category with a Cisinski model structure where weak equivalences are stable under filtered colimits. Let $A \to B$ be a monomorphism of presheaves. Let $\calS$ be a collection of intermediate subobjects $A \subseteq X \subseteq B$, and let $\bar \calS$ denote the closure of $\calS$ under unions. Suppose that the following conditions hold:
\begin{enumerate}
    \item $\bar \calS$ is closed under intersections. (In other words, $B = \cup \calS$ and for $X,Y \in \calS$, we have that $X \cap Y$ is a union of elements in $\calS$.)
    \item For every $X \in \calS$, the map $A \to X$ is an equivalence.
\end{enumerate}
Then the map $A \to U$ is an equivalence for every $U \in \bar \calS$. In particular, the map $A \to B$ is an equivalence.
\end{lem}
\begin{proof}
Because weak equivalences are stable under filtered colimits, it will suffice to treat the case where $\calS$ is finite. Note that $\bar \calS$ is a finite distributive sublattice of the subobject lattice of $B$. We show, by induction on the structure of the lattice $\bar S$, that the map $A \to U$ is an equivalence for every $U \in \bar \calS$. This will suffice since $B \in \bar \calS$. In the base case, $U$ is $\cup$-irreducible, i.e. $U \in \calS$. In this case the map $A \to U$ is an equivalence by hypothesis. Otherwise, we may write $U = V \cup W$ where $V,W \subsetneq U$. We have a commutative diagram as follows:
\begin{equation}\label{eq:dist-diag}
    \begin{tikzcd}
        A \ar[dr,tail] \ar[r,tail] \ar[rr,tail,bend left] &
        V \cap W \ar[r,tail] \ar[d,tail] \ar[dr,phantom,"\ulcorner"{description, very near end}] &
        V \ar[d,tail] \\
        &
        W \ar[r,tail] &
        U
    \end{tikzcd}
\end{equation}
The square is a pushout of monomorphsims of presheaves, hence a homotopy pushout. 
By induction, the maps $A \to V,W, V \cap W$ are equivalences. It follows that the map from $A$ to the homotopy pushout $U$ is likewise a homotopy equivalence.
\end{proof}

\begin{obs}\label{obs:filtered}
Let $N \in \nats \cup \{\omega\}$. As $\Cat_N^\inc$ (resp. $\sCat_N$) is a localization of $\Psh(\Theta_N)$ at compact objects, the $L_{\Cat_N^\inc}$-acyclic (resp. $L_{\sCat_N}$-acyclic) morphisms are stable under filtered colimits, and the inclusions $\sCat_N \to \Cat_N^\inc \to \Psh(\Theta_N)$ preserve filtered colimits.
\end{obs}

\subsection{Suspensions}\label{subsec:susp}
One main ingredient needed to understand $\Cat_N^\inc$ as a localization of $\Psh(\Theta_N)$ is an understanding of suspension, to which we now turn.

\begin{lem}\label{lem:point}
Let $N \in \nats \cup \{\omega\}$. Let $X \in \Psh(\Theta_N)$. The unique map $X \to [0]$ induces a map $\tilde \Sigma X \to \tilde \Sigma [0] = \Sigma [0] = [1]$. For any $Y \in \Psh(\Theta_{1+N})$ and either map $d_i : [0] \rightrightarrows [1]$, we let $d_i t : Y \rightrightarrows [1]$ be the composite maps. Then the space $\Nat(Y, \tilde \Sigma X) \times_{\Nat(Y, [1])} \{d_i t\}$ is contractible.
\end{lem}
\begin{proof}
First observe that as a functor of $Y$, the space $\Nat(Y, \tilde \Sigma X) \times_{\Nat(Y, [1])} \{d_i t\}$ carries colimits to limits. This reduces us to the case where $Y = \theta$ is representable. In this case, $\Nat(\theta, \tilde \Sigma X) \times_{\Nat(\theta, [1])} \{d_i t\}$ preserves contractible colimits in $X$ as a functor $\Psh(\Theta_N) \to \Spaces$. Moreover, this functor carries the empty presheaf to a point. So when we lift to a functor $\Psh(\Theta_N) \to \Spaces_\ast$, we obtain a colimit-preserving functor of $X$. This allows us to reduce to the case where $X = \zeta$ is also representable. In this case, we have $\tilde \Sigma \zeta = \Sigma \zeta$, and we are reduced to a computation in the category $\Theta_{1+N}$, where it is obviously true that $\Nat(\theta, \Sigma \zeta) \times_{\Nat(\theta, [1])} \{d_i t\}$ is contractible.
\end{proof}

\begin{lem}\label{lem:sig-sig}
Let $N \in \nats \cup \{\omega\}$. For $X,Y \in \Psh(\Theta_N)$, the natural map $\amalg_{\phi : \partial[1] \to [1]} \Nat(\tilde \Sigma X, \tilde \Sigma Y) \times_{\Nat(\Sigma \emptyset, \Sigma [0])} \{\phi\} \to \Nat(\tilde \Sigma X, \tilde \Sigma Y)$ is an equivalence. Moreover, when $\phi$ is natural inclusion, corresponding summand is $\Nat(\tilde \Sigma X, \tilde \Sigma Y) \times_{\Nat(\Sigma \emptyset, \Sigma [0])} \{\phi\} = \Nat_{//}(\tilde \Sigma_{//} X, \tilde \Sigma_{//} Y) \cong \Nat(X,Y)$. When $\phi = td_i$ is constant at an endpoint, the summand is contractible. When $\phi$ swaps the order of the endpoints, the summand is contractible if $X = \emptyset$ and empty otherwise.
\end{lem}
\begin{proof}
The first statement follows from the fact that $\Nat(\Sigma \emptyset, \Sigma [0])$ is discrete. The identification of the fiber when $\phi$ is the natural inclusion follows by definition. The identification when $\phi$ is constant follows from \cref{lem:point}. When $\phi$ is the swap map, the case where $X = \emptyset$ likewise follows from \cref{lem:point}, after decomposing $\partial [1] = [0] \amalg [0]$ and observing that the composite maps from $[0]$ are constant. When $\phi$ is the swap map and $X$ is nonempty, any map would induce a similar map in $\sCat_N$ by applying $\pi_0$ levelwise to all the presheaves involved. But in $\sCat_N$ we know that $\Hom_{\Sigma_s X}(0,1) = X$ is nonempty when $X$ is nonempty whereas $\Hom_{\Sigma_s Y}(1,0) = \emptyset$. So there can be no map from the former to the latter, and this fiber is empty as desired.
\end{proof}

\begin{lem}\label{lem:base}
Let $N \in \nats \cup \{\omega\}$. Let $X \in \Psh(\Theta_N)$. Then $\tilde \Sigma X$ is right-orthogonal to the 1-dimensional spine inclusions $[1] \vee \cdots \vee [1] \to [n]$, and also to the maps $S^d \to [0]$.
\end{lem}
\begin{proof}
Let $f : A \to B$ be one of the maps in question. First observe that if $C \in \sCat_N$, then $C$ is right orthogonal to $A \to B$. The unique map $X \to [0]$ induces a map $\tilde \Sigma X \to [1]$, and we have a commutative square as follows:
\begin{equation*}
    \begin{tikzcd}
        \Nat(B, \tilde \Sigma X) \ar[r] \ar[d] & \Nat(A, \tilde \Sigma X) \ar[d] \\
        \Nat(B, [1]) \ar[r] & \Nat(A, [1])
    \end{tikzcd}
\end{equation*}
We would like to show that the top rightward map is an equivalence. Because $[1] \in \sCat_N$, the bottom rightward map is an equivalence. So it will suffice to check that, for each map $d : B \to [1]$, the induced map of fibers 
\begin{equation}\label{eqn:map}
    \Nat(B, \tilde \Sigma X) \times_{\Nat(B, [1])} \{d\} \to \Nat(A, \tilde \Sigma X) \times_{\Nat(A, [1])} \{fd\}
\end{equation}
is an equivalence.

Consider first the case where $f : A \to B$ is the map $S^d \to [0]$. In this case, there are two maps $d_i t: B \rightrightarrows [1]$, each of which factors through one of the maps $d_i : [0] \rightrightarrows [1]$. By \cref{lem:point}, both sides of (\ref{eqn:map}) are contractible, so we have an equivalence as desired.

Consider now the case where $f : A \to B$ is a spine inclusion $[1]^{\vee n} \to [n]$. In this case, the maps $[1]^{\vee n} \to [1]$ fall into two classes. There are two constant maps, which factor through the face maps $[0] \rightrightarrows [1]$. In this case, \cref{lem:point} shows that both sides of (\ref{eqn:map}) are contractible just as before. The other maps are non-constant. Each such map $\phi$ is the identity on one of the wedge summands (the $j$th summand, say) and constant on the others. So the codomain of (\ref{eqn:map}) decomposes as
\begin{align*}
    &(\Nat([1], \tilde \Sigma X) \times_{\Nat([0], \tilde \Sigma X)} \cdots \times_{\Nat([0], \tilde \Sigma X)} \Nat([1], \tilde \Sigma X)) \times_{\Nat([1], [1]) \times_{\Nat([0], [1])} \cdots \times_{\Nat([0], [1])} \Nat([1], [1])} \{\phi\} \\
    \cong
    & (\Nat([1], \tilde \Sigma X) \times_{\Nat([1], [1])} \{d_1 t\} \times_{\Nat([0], [1]) \times_{\Nat([0], [1]} \{d_1\}} \cdots \\
    &\cdots \times_{\Nat([0], [1]) \times_{\Nat([0], [1]} \{d_1\}} (\Nat([1], \tilde \Sigma X) \times_{\Nat([1], [1])} \{\id_{[1]}\} \times_{\Nat([0], [1]) \times_{\Nat([0], [1]} \{d_0\}} \cdots \\
    &\cdots \times_{\Nat([0], [1]) \times_{\Nat([0], [1]} \{d_0\}} (\Nat([1], \tilde \Sigma X) \times_{\Nat([0], [1])} \{d_0 t\} \\
    \cong 
    & (\Nat([1], \tilde \Sigma X) \times_{\Nat([1], [1])} \{\id_{[1]}\} 
\end{align*}
Here we have first interchanged limits, and then observed by \cref{lem:point} that each of the other terms appearing in the middle expression is contractible. Thus we are asking whether a certain map
\begin{equation*}
    (\Nat([n], \tilde \Sigma X) \times_{\Nat([n], [1])} \{\phi \} \to (\Nat([1], \tilde \Sigma X) \times_{\Nat([1], [1])} \{\id_{[1]}\} 
\end{equation*}
is an equivalence. Both the domain and codomain manifestly preserve contractible colimits in $X$. Moreover, as $\tilde \Sigma \emptyset = \partial [1]$, we see that both the domain and codomain preserve the initial object. Thus both the domain and codomain preserve all colimits in $X$. So we are reduced to the case where $X = \zeta$ is representable. In this case $\tilde \Sigma X = \Sigma \zeta$ lies in $\Cat_N^\inc$, so the result is true.
\end{proof}

\begin{thm}\label{thm:susp-pres}
Let $N  \in \nats \cup \{\omega\}$. The suspension functor $\tilde \Sigma : \Psh(\Theta_N) \to \Psh(\Theta_{1+N})$ carries $\Cat_N^\inc$ to $\Cat_{1+N}^\inc$ and $\sCat_N$ to $\sCat_{1+N}$.
\end{thm}
\begin{proof}
All of the required statements are of the form ``If $C$ is right-orthogonal to the maps $\tilde \Sigma^m A \to \tilde \Sigma^m B$ for all $m \geq 0$, then $\tilde \Sigma C$ is right-orthogonal to the maps $\tilde \Sigma^{m'} A \to \tilde \Sigma^{m'} B$ for $m' \geq 0$." The cases $m' = 0$ follow from \cref{lem:base}. For the cases $m' \geq 1$, observe that by \cref{lem:sig-sig}, we have identifications $\Nat(\tilde \Sigma^{m'} A, \tilde \Sigma C) \cong \Nat_{//}(\tilde \Sigma_{//}(\Sigma^{m'-1} A), \Sigma_{//} C)_{++} = \Nat(\Sigma^{m'-1} A, C)_{++}$ and similarly for $B$. So the desired orthogonality follows from the hypothesis where $m = m'-1$.
\end{proof}

\begin{cor}\label{cor:susp-res}
Let $N \in \nats \cup \{\omega\}$. The suspension adjunction 
\[\begin{tikzcd} 
    \tilde \Sigma_{//} \ar[r, shift left]: \Psh(\Theta_N) & \Psh(\Theta_{1+N})_{\partial[1] / /[1]} : \tilde \Hom_{/} \ar[l, shift left] 
\end{tikzcd}\] 
restricts to suspension adjunctions
\[\begin{tikzcd}[row sep = 0.5]
    \Sigma_{//} \ar[r, shift left]: \Cat_N^\inc & (\Cat_{1+N}^\inc)_{\partial[1] / /[1]} : \Hom_{/} \ar[l, shift left] \\
    \Sigma_{//}^s \ar[r, shift left]: \sCat_N & (\sCat_{1+N})_{\partial[1] / /[1]} : \Hom^s_{/} \ar[l, shift left]
\end{tikzcd}\]
which are left Kan extended from $\Theta$. In particular, $\Sigma^s_{//}$ agrees with the usual suspension of strict $N$-categories, and $\Sigma_{//}$ agrees with $\Sigma^s_{//}$ when restricted to strict $N$-categories.
\end{cor}
\begin{proof}
By \cref{thm:susp-pres}, $\tilde \Sigma$ preserves $\Cat_N^\inc$ and $\sCat_N$, so that it descends to the categories indicated. Moreover, $\tilde \Sigma$ carries spine inclusions to spine inclusions, and maps $\Sigma^m S^d \to \Sigma^m [0]$ to maps of the same form, so that $\Hom_{/}$ likewise descends to the categories indicated. Thus $\Sigma_{//}$ agrees with $\Sigma^s_{//}$ on strict $N$-categories, since they both agree with $\tilde \Sigma_{//}$.

To see that $\Sigma^s_{//}$ agrees with the usual suspension of strict $N$-categories, observe that both functors preserve colimits (admitting right adjoints given by $\Hom^s_/$ and the usual hom functor respectively) and agree on $\Theta_N$, which is dense in $\sCat_N$.
\end{proof}

\begin{obs}\label{obs:hty-susp}
Let $N \in \nats \cup \{\omega\}$. By construction, the suspension functors $\Sigma : \Cat_N^\inc \to \Cat_{1+N}^\inc$ and $\Sigma^s : \sCat_N \to \sCat_{1+N}$ of \cref{cor:susp-res} preserve contractible colimits, and in particular pushouts. 
Moreover they agree on strict $N$-categories. Hence $\Sigma^s$ preserves homotopy pushouts.
\end{obs}

\subsection{Wedge sums}\label{subsec:wedge}
Besides suspensions, the other main ingredient needed to understand $\Cat_N^\inc$ as a localization of $\Psh(\Theta_N)$ is an understanding of wedge sums, to which we now turn.

\begin{fact}\label{fact:wedge}
Let $N \in \nats \cup \{\omega\}$. Let $A^-, A^+$ be $\Theta$-regular strict $N$-categories, and let $a^-_+ \in A^-$ be a sink and $a^+_- \in A^+$ a source. Then $A^- \vee A^+ = A^- \tensor[_{a^-}]{\vee}{_{a^+}}{^{\sCat_N}} A^+$ has object set $\Ob A^- {}_{a^-_+} \vee^{\Set}_{a^+_-} A^+$, the disjoint union of $\Ob A^+$ and $\Ob A^-$ modulo the identification of $a^-_+$ with $a^+_-$ as the \defterm{wedge point}. Moreover, the inclusions $A^-, A^+ \rightrightarrows A^- \vee A^+$ are full subcategory inclusions. For $a^- \in A^-, a^+ \in A^+$, we have $\Hom_{A^- \vee A^+}(a^-, a^+) = \Hom_{A^-}(a^-, a^-_+) \times \Hom_{A^+}(a^+_-, a^+)$, and, so long as $a^-,a^+$ are not both equal to the wedge point, $\Hom_{A^- \vee A^+}(a^+,a^-) = \emptyset$.
\end{fact}

\begin{lem}\label{lem:refl-for-wedge}
Let $N \in \nats \cup \{\omega\}$. Let $A^-, A^+$ be $\Theta$-regular strict $N$-categories, and let $a^-_+ \in A^-$ be a sink and $a^+_- \in A^+$ a source. Let $A = A^- \tensor[_{a^-}]{\vee}{_{a^+}} A^+$. Let $\calC \subseteq \Theta_N \downarrow A$ be the full subcategory of those nondegenerate $\theta \to A$ which either factor through $A^-$ or $A^+$, or else hit the wedge point $a$. Then $\calC$ is reflective in $\Theta_N \downarrow A$.
\end{lem}
\begin{proof}
Let $\calB \subseteq \Theta_N \downarrow A$ comprise those $\theta \to A$ which are injective on objects. Note that $\calB$ is reflective and hence final in $\Theta_N \downarrow A$ (using that $A$ is $\Theta$-regular). So it will suffice to show that $\calC$ is reflective in $\calB$.

To see this, first observe that if $F : \theta \to A$ factors through $A^-$ or $A^+$, then the Eilenberg-Zilber factorization of $F$ is its reflection in $\calC$.

Now suppose that $F : \Sigma \theta \to A$ carries the first object $0$ into $a^- \in A^- \setminus \{a\}$ and the second object $1$ into $a^+ \in A^+ \setminus \{a\}$. Then $F$ corresponds to a functor $f : \theta \to \Hom_A(a^-,a^+) = \Hom_{A^-}(a^-, a) \times \Hom_{A^+}(a, a^+)$, by \cref{fact:wedge}. Writing $f^- : \theta \to \Hom_A(a^-,a)$ and $f^+ : \theta \to \Hom_A(a, a^+)$ for the coordinates of $f$, we take Eilenberg-Zilber factorizations $f^\sigma = g^\sigma h^\sigma$, where $h^\sigma : \theta \to \theta^\sigma$ are degeneracies. We obtain functors $h = \langle h^-, h^+ \rangle  : \theta \to \theta^- \times \theta^+$ and $g = g^- \times g^+: \theta^- \times \theta^+ \to \Hom_{A^-}(a^-,a) \times \Hom_{A^+}(a,a^+) = \Hom_A(a^-,a^+)$, and hence functors  $\Sigma \theta \xrightarrow H \Sigma \theta^- \vee^{\sCat_N} \Sigma \theta^+ \xrightarrow G A$. We claim that $H$ exhibits $G \in \calC$ as a reflection of $F$ in $\calC$. Indeed, suppose that $K : \Sigma \zeta^- \vee \Sigma \zeta^+ \to A$ lies in $\calC$. If $K$ factors through $A^-$ or $A^+$, then there can be no map $F \Rightarrow K$ or $G \Rightarrow K$. Otherwise, $K$ hits the wedge point, and so decomposes as $K^- \vee K^+ : \Sigma \zeta^- \vee \Sigma \zeta^+ \to A^- \vee A^+$ where $K^\sigma : \Sigma \zeta^\sigma \to A^\sigma$ hits the wedge point and is nondegenerate, corresponding to $k^- : \zeta^- \to \Hom_{A^-}(a^-,a)$ and $k^+ : \zeta^+ \to \Hom_{A^+}(a,a^+)$, which are nondegenerate because $K \in \calC$. If $L : \Sigma \theta \to \Sigma \zeta^- \vee \Sigma \zeta^+$ is such that $(K^- \vee K^+)L = F$, then it corresponds to maps $l^\sigma : \theta \to \zeta^\sigma$ with $k^\sigma l^\sigma = f^\sigma$. By uniqueness of Eilenberg-Zilber factorizations, there are unique maps $m^\sigma : \zeta^\sigma \to \theta^\sigma$ such that $m^\sigma l^\sigma  = h^\sigma$ and $k^\sigma = g^\sigma m^\sigma$. These assemble into a map $M : \Sigma \zeta^- \vee \Sigma \zeta^+ \to \Sigma \theta^- \vee \Sigma \theta^+$ which factors $L$ through $H$ as desired. For uniqueness of this map, note that it must carry the wedge point to the wedge point, and then the action on $\theta^\sigma$ is determined separately by the uniqueness of Eilenberg-Zilber decompositions.

Checking the universal property when the domain of $K$ has more objects is similar, as is the construction when the domain of $F$ has more objects.
\end{proof}

\begin{lem}\label{lem:wedge-lke}
Let $N \in \nats \cup \{\omega\}$. Let $A^-, A^+$ be $\Theta$-regular strict $N$-categories, and let $a^-_+ \in A^-$ be a sink and $a^+_- \in A^+$ a source. Let $A = A^- \tensor[_{a^-}]{\vee}{_{a^+}} A^+$. Let $\calC \subseteq \Theta_N \downarrow A$ and $W : \calC \to \Psh(\Theta_N)$ be as in \cref{lem:refl-for-wedge}. Let $\calD \subseteq \calC$ comprise the (nondegenerate) maps which either factor through $A^-$ or through $A^+$, and let $X : \calD \to \Psh(\Theta_N)$ be the restriction of $W$. Then the identity transformation exhibits $W$ as the left Kan extension of $X$ along the inclusion $i : \calC \to \calD$.
\end{lem}
\begin{proof}
Let $(\theta \to A) \in \calC \setminus \calD$. The formula for pointwise left Kan extensions tells us that $\Lan_i X(\theta \to A) = \varinjlim^{\Psh(\Theta_N)}((i \downarrow (\theta \to A)) \to \calD \xrightarrow X \Psh(\Theta_N))$. The category $i \downarrow (\theta \to A)$ decomposes as a union of simplicial sets $(\calD \downarrow \theta^-) \vee (\calD \downarrow \theta^+)$. Since this is a pushout of monomorphisms of simplicial sets, it induces an equivalence 
\begin{align*}
    \Lan_i X(\theta \to A) 
    &= \varinjlim((\calD \downarrow \theta^-) \to \calD \to \Psh(\Theta_N)) \vee \varinjlim((\calD \downarrow \theta^+) \to \calD \to \Psh(\Theta_N)) \\
    &= \yo \theta^- \vee^{\Psh(\Theta_N)} \yo \theta^+ \\
    &= W(\theta \to A)
\end{align*}
as desired.
\end{proof}

\begin{thm}\label{lem:wedge}
Let $N \in \nats \cup \{\omega\}$. Let $A^-, A^+$ be $\Theta$-regular strict $N$-categories, and let $a^-_+ \in A^-$ be a sink and $a^+_- \in A^+$ a source. Let $A = A^- \tensor[_{a^-}]{\vee}{_{a^+}} A^+$ be the wedge sum. Then the strict pushout
\begin{equation*}
    \begin{tikzcd}
        \ast \ar[r,"a^-"] \ar[d,"a^+"] \ar[dr,phantom,"\ulcorner"{description, very near end}] & 
        A^- \ar[d] \\
        A^+ \ar[r] & A
    \end{tikzcd}
\end{equation*}
is a homotopy pushout.
\end{thm}
\begin{proof}
Let $A = A^- \vee^{\sCat_N} A^+$, and let $a$ denote the wedge point. By density, the functor $U : \Theta_N \downarrow A \to \Psh(\Theta_N)$, $(\theta \to A) \mapsto \yo \theta$, has colimit $\nerve A$.

As in \cref{lem:refl-for-wedge}, let $\calC \subseteq \Theta_N \downarrow A$ denote the full subcategory of those maps $\theta \to A$ which are nondegenerate and are either contained in $A^-$ or $A^+$, or else hit the wedge point. By \cref{lem:refl-for-wedge}, the inclusion $\calC \to \Theta_N \downarrow A$ is reflective and hence final, so the functor $V : \calC \to \Psh(\Theta_N)$, obtained by restriction of $U$, likewise has colimit $\nerve A$.

Let $W : \calC \to \Psh(\Theta_N)$ carry $F : \theta \to A$ to $\yo \theta$, if $f$ factors through $A^-$ or $A^+$, and to $\yo \theta^- \vee^{\Psh(\Theta_N)} \yo \theta^+$ otherwise, where $\theta = \theta^- \vee \theta^+$ is the decomposition of $\theta$ at the wedge point. Then there is a natural inclusion $W \to V$, which is levelwise $L_{\Cat_N^\inc}$-acyclic because it is in fact one of the basic wedge inclusion maps generating the localization $L_{\Cat_N^\inc}$ (\cref{def:spine}). Therefore, the induced map $\varinjlim^{\Psh(\Theta_N)} W \to \nerve A$ is $L_{\Cat_N^\inc}$-acyclic.

Let $\calD \subseteq \calC$ be defined as in \cref{lem:wedge-lke}, as well as $X: \calD \to \Psh(\Theta_N)$. By \cref{lem:wedge-lke}, we have an equivalence $\varinjlim^{\Psh(\Theta_N)} W = \varinjlim^{\Psh(\Theta_N)} X$.

Now, the category $\calD$ admits a decomposition $\calD = \calD^- \vee \calD^+$, where $\calD^\sigma \subseteq \Theta_N \downarrow \nerve A^\sigma$ comprises the nondegenerate maps. This is a pushout of monomorphisms of simplicial sets, and so induces an equivalence $\varinjlim^{\Psh(\Theta_N)} X = \varinjlim^{\Psh(\Theta_N)} X^- \vee^{\Psh(\Theta_N)} \varinjlim^{\Psh(\Theta_N)} X^+$, where $X^\sigma$ denotes the restriction of $X$ to $\calD^\sigma$.

Finally, $\calD^\sigma$ is reflective in $\Theta_N \downarrow \nerve A^\sigma$. As $\Theta_N$ is dense in $\Psh(\Theta_N)$, this implies that $\varinjlim^{\Psh(\Theta_N)} X^\sigma = \nerve A^\sigma$.

Chaining everything together, we have that the canonical map $\nerve A^- \vee^{\Psh(\Theta_N)} \nerve A^+ \to \nerve A$ is $L_{\Cat_N^\inc}$-acyclic. Since $L_{\Cat_N^\inc}$ preserves colimits, this yields an equivalence $A^- \vee^{\Cat_N^\inc} A^+ \to A$ as desired.
\end{proof}

\section{Some miscellaneous generalities}\label{sec:misc}
In this short subsection we collect a few generalities on posets which will be used to analyze the poset $\Decomp(\mu)$ of composites in \cref{sec:cs}. The discussion in this section is independent of any considerations about torsion-free complexes or $N$-categories. \cref{subsec:posets} feeds into \cref{cor:forest-sd-ideal} below, and \cref{subsec:a+} feeds directly into \cref{thm:sc-main} below.

\subsection{Some generalities on posets}\label{subsec:posets}

In this subsection, we analyze the homotopy type of the poset $\lin(P)$ (\cref{def:linpre}) of linear preorders refining a given preorder $P$. We show that it is contractible so long as $P$ is not an equivalence relation. The application we have in mind is \cref{cor:forest-sd-ideal} below, where we take $P$ to be the preorder $\myPoset_k(\mu)$ of \cref{def:pos} above.

\begin{Def}\label{def:linpre}
    Let $L$ be a preorder. We say that $L$ is a \defterm{linear preorder} if for all $a, b \in L$, $a \leq b$ or $b \leq a$ (or both). Thus a linear order is a linear preorder which is also a partial order. We say that $L$ is \defterm{codiscrete} if for all $a,b$ we have $a \leq b$. So there is a unique codiscrete preorder $\bot$ on any set, and it is linear. We say that $L$ is \defterm{discrete} if $a \leq b \Leftrightarrow a = b$. Note that $L$ is an equivalence relation if and only if its posetal reflection is discrete.

    Let $P,P'$ be preorders on the same underlying set $S$. We say that $P'$ \defterm{refines} $P$ if $a \leq^P b \Rightarrow a \leq^{P'} b$. The set of preorders on a set $S$ form a partial order under reverse refinement, and the codiscrete preorder is at the bottom.
    
    Let $P$ be a preorder. We let $\linbot(P)$ denote the poset of linear preorders refining $P$, ordered by reverse refinement. We let $\linbot(P) = \lin(P) \setminus \{\bot\}$ denote the subposet of noncodiscrete linear preorders refining $P$.

    Let $P$ be a preorder. We let $\DCbot(P)$ denote the poset of downward-closed subsets of $P$, ordered by inclusion. We let $\DC(P) = \DCbot(P) \setminus \{\emptyset\}$ denote the subposet of nonempty downward-closed subsets of $P$.

    We denote by $\sd P$ the barycentric subdivision of a preorder (or category, or simplicial set) $P$. Recall that this is the poset of nonempty subsets $p_0 < \cdots < p_n$ of $P$ which are linearly ordered by $P$, ordered by containment. Recall also that $|\sd P| \simeq |P|$ for all posets $P$.
\end{Def}

So if $L,L' \in \lin(P)$, we have $L \leq L'$ if $a \leq^{L'} b \Rightarrow a \leq^{L} b$. Since $L,L'$ are linear preorders, this amounts to saying that the $L'$-isomorphism classes partition $P$ more finely than the $L$-isomorphism classes.

\begin{lem}\label{lem:iso-sd}
    Let $P$ be a poset. Then $\lin(P)$ is isomorphic to $\sd (\DC(P))$.
\end{lem}
\begin{proof}
    An element $L \in \lin(P)$ induces an equivalence relation on $P$ and then linearly orders the equivalence classes $E_1 < E_2 < \cdots < E_n$ in such a way as to refine the ordering on $P$. From $L$ we may extract a chain of nonempty proper ideals in $P$ given by $E_1 \subset E_1 \cup E_2 \subset \cdots \subset E_1 \cup \cdots \cup E_{n-1}$. In other words, we have extracted a proper chain in $\DC(P)$, which is nonempty if and only if $L$ is not codiscrete. This construction is an isomorphism of posets.
\end{proof}

We thank Gregory Arone for the proof of the following lemma \cite{arone}. As pointed out by Sam Hopkins and Benjamin Steinberg in the comments at \cite{arone}, the proof in fact applies to any finite non-boolean distributive lattice (not just to $\DC(P)$), and is closely related to a classic application of Rota's crosscut theorem in poset combinatorics.
\begin{lem}\label{lem:dist-lat'}
   Let $P$ be a preorder which is not an equivalence relation. Then $\DC(P)$ is contractible.
\end{lem}
\begin{proof}
    Homotopy types are invariant under equivalence of categories, so we may assume for simplicity that $P$ is a poset.
    Let $M$ be the set of minimal elements in $P$ (a discrete subposet of $P$). Since $P$ is not discrete, we have $M \subsetneq P$, and so $\DC(M) \cup \{M\} \subseteq \DC(P)$. As $\{M\}$ gives a terminal object in $\DC(M) \cup \{M\}$, this poset is contractible. So it will suffice to show that $\DC(M) \cup \{M\}$ is coreflective in $\DC(P)$. To see this, note that for any $p \in P$, there is some $m \in M$ with $m \leq p$. So for $I \in \DC(P)$, the set $M \cap I$ is nonempty and gives the coreflection of $I$ in $\DC(M) \cup \{M\}$.
\end{proof}

\begin{cor}\label{cor:non-disc-contr}
    Let $P$ be a preorder which is not an equivalence relation. Then $\lin(P)$ is contractible.
\end{cor}
\begin{proof}
    By \cref{lem:iso-sd}, $\lin(P)$ is isomorphic to $\sd (\DC (P))$. By \cref{lem:dist-lat'}, $\DC(P)$, and hence also $\sd(\DC(P))$, is contractible. Therefore $\lin(P)$ is contractible as well.
\end{proof}

\begin{rmk}\label{rmk:disc-sphere}
    Let $P$ be an equivalence relation. Then $\lin (P)$ has the homotopy type of $S^{d-2}$, where $d$ is the number of equivalence classes. For $|\lin (P)| \simeq |\sd(\DC(P))| \simeq |\DC(P)|$, and $\DC(P)$ is the poset of nonempty proper subsets of $P$-equivalence classes. So $\DC(P)$ is isomorphic to the $d$-cube with top and bottom removed, and so has the homotopy type of $S^{d-2}$.
\end{rmk}

\subsection{A variant of Quillen's Theorem A}\label{subsec:a+}
In this subsection, we give a variant of Quillen's Theorem A which will be used to relate the homotopy type of $\Decompat{[k-1,n-1]}(\mu)$ to that of $\Decompat{\{k\}}(\mu)$ (see \cref{def:decomp'}) in the proof of \cref{thm:sc-main} below.

\begin{rmk}
    One key tool we use is Quillen's Theorem A \cite{quillen}. We will in fact only apply it to posets, in which setting it usually goes by the name \emph{the Quillen fiber lemma}. Our use of Quillen's Theorem A is further restricted to the case of a cocartesian fibration, where the hypotheses become simpler to state. That is, we use the following statement: If $F: \calC \to \calD$ is a cocartesian fibration (of $\infty$-categories), and if the classifying spaces of the fibers $|F\inv(D)|$ are contractible for each $D \in \calD$, then the induced map on classifying spaces $|F| : |\calC| \to |\calD|$ is a homotopy equivalence. 
\end{rmk}

\begin{lem}\label{lem:init-cofinal}
    Let $I$ be a contractible $\infty$-category and let $I^\triangleleft$ be $I$ with a new initial object adjoined. Then the inclusion $I \to I^\triangleleft$ is cofinal.
\end{lem}
\begin{proof}
    By the Joyal/Lurie Quillen’s Theorem A \cite[Theorem 4.1.3.1]{htt}, it suffices to show that for each $i \in I^\triangleleft$, the slice category $i/I$ is contractible. When $i \in I$, this category has an initial object given by $i$ itself. When $i$ is the new initial object, this category is isomorphic to $I$, which is contractible by hypothesis.
\end{proof}

\begin{lem}\label{lem:special}
    Let $U : J \to I$ be a cocartesian fibration. Suppose that
    \begin{enumerate}
        \item $I$ has a strict initial object $\bot$;
        \item The classifying space $|I \setminus \{\bot\}|$ is contractible;
        \item For each $i \in I \setminus \{\bot\}$, the classifying space $|U\inv(i)|$ of the fiber is contractible.
    \end{enumerate}
    Then the classifying space $|J|$ is contractible.
\end{lem}
\begin{proof}
    By \cite{gepner-haugseng-nikolaus}, $J$ is the lax colimit of the functor $\phi: I \to \Cat_\infty$ which $U$ classifies. Hence the classifying space $|J|$ is the colimit of the functor $|\phi| : I \to \Cat \to \Spaces$ obtained by postcomposing the ``classifying space” functor. 
    The inclusion $I \setminus \{\bot\} \to I$ is cofinal by (1) and \cref{lem:init-cofinal}. Therefore $|J| = \colim |\phi| = \colim |\phi|_{I \setminus \{\bot\}}|$. The functor $\phi|_{I \setminus \{\bot\}}$ is constant at the one-point space by (3), and so its colimit is the geometric realization $|I \setminus \{\bot\}|$, which is contractible by (2).
\end{proof}

\section{Spaces of composites}\label{sec:cs}

In this section, we work with torsion-free complexes of dimension $\leq N = \omega$. Let $P$ be a torsion-free complex and $\mu$ an $n$-cell in $P$. In this section, we define (\cref{def:decomp}) a poset $\Decompbot(\mu) = \mu \downarrow \Theta \downarrow_{\nd} \Free P$ whose elements are ``$\Theta$-shaped decompositions of the cell $\mu$ as a composite of cells". This poset has a bottom element given by $\mu$, trivially subdivided as the composite of itself. The main goal of this section is to show (Theorem A / \cref{thm:sc-main}) that when we delete the bottom element, the remaining poset $\Decomp(\mu) := \Decompbot(\mu) \setminus \{\mu\}$ is contractible. This result will be a key input to \cref{prop:level}, and thus to \cref{thm:mainthm} below. It may also be regarded as an analog of the main result of \cite{columbus} or \cite{horr} in the case $N = 2$.

The proof of \cref{thm:sc-main} proceeds by projecting $\Decomp(\mu)$ down to a smaller poset $\Decompatbot{[0,k-1]}(\mu)$ (\cref{def:decomp}), roughly given by projecting onto the positive $k$-boundary of $\mu$, and using Quillen's Theorem A to reduce to a study of the fibers of this projection.\footnote{We do point out that in order to reduce to the study of the fibers rather than the lax fibers of this projection, we need to know that the projection is a cocartesian fibration (\cref{lem:cocart''}), and for this we already need to use (\cref{lem:cocart}) some details coming from $P$ being a torsion-free complex.} Only the fiber over $\mu$ itself is not trivially contractible, and this fiber is given by another sub-poset $\Decompat{[k-1,n-1]}(\mu)$. The choice of $k$ is given by \cref{lem:exists-min}. The intuition comes from the case when $k = n$ is maximal. In this case, $\Decompat{\{n-1\}}(\mu)$ consists of those $\Theta$-decompositions of $\mu$ which are ``linear" in the sense that they only use $\circ_{n-1}$ composition, and not $\circ_i$ composition for $i \leq n-2$. This poset can be completely analyzed using the techniques of \cite{forest-thesis} (\cref{prop:pos-iso}), plus some elementary manipulations  (\cref{subsec:posets}) and the result (\cref{cor:forest-sd-ideal}) is that it is contractible so long as there is at least one source-target dependency among the $n$-atoms of $P$ (i.e. the relation $\triangleleft_{n-1}$ is not discrete). Thus in most cases, we take $k = n-1$ and we are done. But when $\triangleleft_{n-1}$ is discrete, we find that $\Decompat{\{n-1\}}(\mu)$ is in fact not contractible (\cref{rmk:eq-sphere}). There is a workaround, though -- we keep decreasing $k$ until (roughly) we find $k$ such that $\triangleleft_{k-1}$ is not discrete (more precisely, we need the preorder $\myPoset_k(\mu)$ of \cref{def:pos} to not be an equivalence relation; such a $k$ exists by \cref{lem:exists-min}). An variant of Quillen's Theorem A (\cref{subsec:a+}) reduces us to the study of the poset $\Decompat{\{k\}}(\mu)$ (\cref{def:decomp}), which can be shown to be contractible by similar methods (\cref{cor:forest-sd-ideal}). 

\subsection{Posets of composites}
In this subsection, we define (\cref{def:decomp}) the poset $\Decompbot(\mu)$ (resp. $\Decomp(\mu)$) of ``decompositions (resp. proper decompositions) of $\mu$" and some auxiliary posets $\Decompat{S}(\mu), \Decompatbot{S}(\mu)$ (\cref{def:decomp'}), as well as certain functors between them (\cref{obs:u}). The auxiliary categories $\Decompat S (\mu)$ and $\Decompatbot S (\mu)$ require a discussion of auxiliary categories $\Theta_S$ (\cref{def:theta-s}).

\begin{Def}\label{def:decomp}
    Let $P$ be a torsion-free complex with an $n$-cell $\mu$. Recall (\cref{def:pos}) that $\overline{\{\mu\}} \subseteq P$ denotes the smallest subcomplex containing $\mu$.

    Let $\Decompbot(\mu)$ denote the poset of subcategories $\theta \subseteq \Free (\overline{\{\mu\}})$ such that $\mu \in \theta$ and $\theta$ is isomorphic to a category in $\Theta_n$. Let $\Decomp(\mu) = \Decompbot(\mu) \setminus \{\mu\}$ be obtained by deleting the bottom element, which is $\mu$ itself. 
\end{Def}

We regard $\Decompbot(\mu)$ as the ``poset of ways to decompose $\mu$ as a composite of cells", and $\Decomp(\mu) \subset \Decompbot(\mu)$ as the ``poset of ways to nontrivially decompose $\mu$ as a composite of smaller cells". The classifying space $|\Decompbot(\mu)|$ (resp. $|\Decomp(\mu)|$) is correspondingly the ``space of decompositions (resp. nontrivial decompositions) of $\mu$". Obviously $|\Decompbot(\mu)| \simeq \ast$ because of the initial object $\mu$. We will show that $|\Decomp(\mu)| \simeq \ast$ in \cref{thm:sc-main}. The proof will involve fibering $\Decompbot(\mu)$ over various subposets $\Decompatbot{[k,l]}(\mu)$, to whose description we now turn.

\begin{Def}\label{def:theta-s}
    Let $S \subseteq \nats$ and $n \in \nats \cup \{\omega\}$. Then let $\Theta_{n,S} \subseteq \Theta_n$ be the full subcategory whose objects are assembled from their generating cells using $\circ_i$ composition only for $i \in S$. 
\end{Def}

\begin{eg}
    Let $n \in \nats \cup \{\omega\}$. Then $\Theta_{n,[0,n-1]} = \Theta_n$. See \cref{fig:theta-s} for more examples.
\end{eg}

\begin{rmk}
    We will only use \cref{def:theta-s} in the case where $S = [k, l] \subseteq [0,n-1]$ is a finite interval.
\end{rmk}

\begin{rmk}
    For $S$ finite we may describe $\Theta_{n,S}$ inductively as follows. In the base case, $\Theta_{n,\emptyset} = \globe_{\leq n} := \Theta_n \cap \globe = \{\globe_k \mid k \leq n\}$ is the globe category, since globes do not have any nontrivial composition defined among their atoms. Inductively, let $S^- = \{i \in \nats \mid i+1 \in S\}$ and $S^+ = \{i+1 \mid i \in S\}$. Then either $S = (S^-)^+$ or $S = (S^-)^+ \cup \{0\}$. In the first case, $\Theta_{n,(S^-)^+}$ consists of the suspensions of elements of $\Theta_{n-1,S^-}$. In the second case, $\Theta_{n,(S^-)^+ \cup \{0\}}$ is the closure of $\Theta_{n,(S^-)^+}$ under wedge sums.

    Equivalently, in the wreath product perspective, we have $\Theta_{n,S^+} = \Theta_{n-1,S} \wr \Delta_{\leq 1}$ and $\Theta_{n, \{0\} \cup S^+} = \Theta_{n-1,S} \wr \Delta$.
    
    Equally, in the usual notation for objects of $\Theta_n$ (as in \cref{fig:theta-s}), the subcategory $\Theta_{n,S}$ constrains in which dimensions numbers other than $0$ or $1$ may appear -- for dimensions not in $S$, only $0$ or $1$ is allowed, while for dimensions in $S$ all numbers are allowed.
\end{rmk}

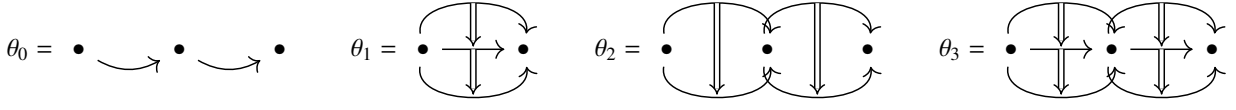
\begin{figure}
    \centering
\begin{equation*}
\theta_0 = 
\begin{tikzcd}
    \bullet \ar[r,bend right] & \bullet \ar[r, bend right] & \bullet
\end{tikzcd}
\qquad
\theta_1 = 
\begin{tikzcd}
    \bullet \ar[r,bend left=10em,""{name=A,inner sep = 0, outer sep = 0}] \ar[r,""{name=B,inner sep = 0, outer sep = 0}] \ar[r,bend right=10em,""{name=C,inner sep = 0, outer sep = 0}] & \bullet \ar[from=A,to=B,Rightarrow] \ar[from=B,to=C,Rightarrow]
\end{tikzcd}
\qquad 
\theta_2 = 
\begin{tikzcd}
    \bullet \ar[r,bend left=10em,""{name=A,inner sep = 0, outer sep = 0}] 
    \ar[r,bend right=10em,""{name=C,inner sep = 0, outer sep = 0}] & 
    \bullet \ar[r,bend left=10em,""{name=D,inner sep = 0, outer sep = 0}] 
    \ar[r,bend right=10em,""{name=F,inner sep = 0, outer sep = 0}] &
    \bullet
    \ar[from=A,to=C,Rightarrow] \ar[from=D,to=F,Rightarrow]
\end{tikzcd}
\qquad
\theta_3 = 
\begin{tikzcd}
    \bullet \ar[r,bend left=10em,""{name=A,inner sep = 0, outer sep = 0}] 
    \ar[r,""{name=B,inner sep = 0, outer sep = 0}] 
    \ar[r,bend right=10em,""{name=C,inner sep = 0, outer sep = 0}] & 
    \bullet \ar[r,bend left=10em,""{name=D,inner sep = 0, outer sep = 0}] 
    \ar[r,""{name=E,inner sep = 0, outer sep = 0}] 
    \ar[r,bend right=10em,""{name=F,inner sep = 0, outer sep = 0}] &
    \bullet
    \ar[from=A,to=B,Rightarrow] \ar[from=B,to=C,Rightarrow] \ar[from=D,to=E,Rightarrow] \ar[from=E,to=F,Rightarrow]
\end{tikzcd}
\end{equation*}
    \caption{$\theta_0 = [2]\in \Theta_{1,\{0\}} = \Theta_1$, $\theta_1 = [1 | 2] \in \Theta_{2,\{1\}}$, $\theta_2 = [2 | 1,1] \in \Theta_{2,\{0\}}$, $\theta_3 = [2 | 2,2] \in \Theta_{2,\{0,1\}} = \Theta_2$}
    \label{fig:theta-s}
\end{figure}

\begin{obs}\label{obs:pre-u}
    Let $S \subseteq \nats$ and $n \in \nats \cup \{\omega\}$. 
    \begin{enumerate}
        \item If $S \subseteq T$, then there is an inclusion $i_{n,S}^{n,T} : \Theta_{n,S} \to \Theta_{n,T}$. If $S$ is an initial segment of $T \subseteq \nats$, then there is also a right adjoint $R_{n,T}^{n,S} : \Theta_{n,T} \to \Theta_{n,S}$, given by ``merging" all $\circ_i$ composites for $i \in T \setminus S$. For example, in \cref{fig:theta-s}, $R_{2,\{0,1\}}^{2,\{0\}}$ carries $\theta_3$ to $\theta_2$, and the counit inclusion $\theta_3 \to \theta_2$ is the identity on objects, carrying the generating 2-cells to composite 2-cells.
        \item If $m \leq n$, then the inclusion $i_{m,S}^{n,S} : \Theta_{m,S} \to \Theta_{n,S}$ admits a \defterm{deformation retraction} $\partial_{n,S}^{m,S}$, which takes the $\partial_m$ boundary of a cell $\theta$. Recall from \cref{notation:strict} that $\partial_m$ means the \emph{forward} $m$-dimensional boundary (although the choice of forward versus backward is arbitrary, we will use it consistently in the following). By ``deformation rertraction", we mean that $\partial_{n,S}^{m,S} \circ i_{m,S}^{n,S}$ is the identity and that we have a natural transformation $i_{m,S}^{n,S} \circ \partial_{n,S}^{m,S} \Rightarrow \id$ which restricts to the identity when we precompose $i_{m,S}^{n,S}$. For example, in \cref{fig:theta-s}, $\partial_{2,\{0\}}^{1,\{0\}}(\theta_2) = \partial_{2,\{0,1\}}^{1,\{0\}}(\theta_3) = \theta_0$; the inclusion is the identity on objects and picks out the final 1-cells in the hom-categories.
    \end{enumerate}
\end{obs}

\begin{Def}\label{def:decomp'}
    Let $P$ be a torsion-free complex with an $n$-cell $\mu$, and let $S \subseteq \nats$. We let $\Decompatbot S (\mu) \subseteq \Decompbot(\mu)$ be the full subposet of $\theta \subseteq \Free (\overline{\{\mu\}})$ containing $\mu$, such that $\theta$ is isomorphic to a category in $\Theta_{n,S}$. We let $\Decompat S (\mu) = \Decompatbot S (\mu) \setminus \{\mu\}$ be obtained by deleting the bottom element $\mu$.
\end{Def}

\begin{obs}\label{obs:u}
    Let $P$ be a torsion-free complex with an $n$-cell $\mu$, and let $S \subseteq \nats$. From \cref{obs:pre-u}, obtain the following:
    \begin{enumerate}
        \item If $S$ is an initial segment of $T \subseteq \nats$, then the adjunction $i_{n,S}^{n,T} : \Theta_S {}^\to_\leftarrow \Theta_T : R_{n,T}^{n,S}$ induces an adjunction $j_{S}^{T} : \Decompatbot S (\mu) {}^\to_\leftarrow \Decompatbot T (\mu) : U_{T}^{S}$ where $j_{S}^{T}$ is the inclusion functor. The functor $U_T^S$ composes together all of the $\circ_i$ composites in $\theta$ for $i \in T \setminus S$.
        \item If $m \leq n$, then the deformation retraction $\partial_{n,S}^{m,S} : \Theta_{n,S} {}^\to_\leftarrow \Theta_{m,S}: i_{m,S}^{n,S}$ induces a functor $D_{n,S}^{m,S} : \Decompatbot S (\mu) \to \Decompatbot S (\partial_m \mu)$, which takes the $\partial_m$ boundary of a $\Theta$-cell.
    \end{enumerate}
\end{obs}

\subsection{Posets of composites in a torsion-free complex}

In this subsection, we analyze the isomorphism type of the poset $\Decompat{\{k\}}(\mu)$ (\cref{def:decomp'}) using the theory of \cite{forest-thesis}. We show that (\cref{lem:forest-order}), when the preorder $\myPoset_i(\mu)$ is an equivalence relation for $i > k$, the poset $\Decompat{\{k-1\}}(\mu)$ is is isomorphic to the poset $\lin (\myPoset_k(\mu))$ of linear preorders refining $\myPoset_k(\mu)$ (see \cref{def:pos} in this subsection). This will tie in nicely with the work of \cref{subsec:posets} in \cref{cor:forest-sd-ideal} below, to show that if in addition the preorder $\myPoset_k(\mu)$ is \emph{not} an equivalence relation, then this poset is contractible. Finding a $k$ satisfying both conditions is always (except in trivial cases) possible by \cref{lem:exists-min}, which we also prove in this subsection. Finally, this subsection contains results (\cref{lem:cocart}, \cref{lem:cocart'}, \cref{lem:cocart''}) stating that certain functors $U$ from the previous subsection (\cref{obs:u}) are cocartesian fibrations, another important input to the proof of \cref{thm:sc-main} below.

\begin{Def}\label{def:extract-preorder}
    Let $P$ be a torsion-free complex with an $n$-cell $\mu$, and let $1 \leq k \leq n$.
    If $\theta \in \Decompat{[0,k-1]}(\mu)$, then write $L_k^\mu(\theta)$ for the preorder on $(\partial_k \mu)_k$ generated by saying that $a$ is equivalent to $b$ if $a,b$ are in the support of the same generator $\bar a = \bar b$ of $\theta$, and $a < b$ if $\bar b \circ_{k-1} \bar a$ is defined. 
\end{Def}

\begin{rmk}
    In \cref{def:extract-preorder}, $L_k^\mu(\theta)$ does not relate $a,b \in (\partial_k \mu)_k$ unless they are in the support of the same generating $k$-cell $c$ of $U_{[0,k-1]}^{[0,k-2]} \theta$. Thus $L_k^\mu(\theta)$ may be regarded as a disjoint union of linear preorders on $c_k$ for each such $c$ -- a \defterm{sublinear preorder}.
\end{rmk}

\begin{prop}\label{prop:pos-iso}
    Let $P$ be a torsion-free complex with an $n$-cell $\mu$. 
    Then $L_n^\mu : \Decompat{\{n-1\}}(\mu) \to \lin (\myPoset_{n}(\mu))$ is an isomorphism of posets.
\end{prop}
\begin{proof}
    First observe that for any $\theta \in \Decompat{\{n\}}(\mu)$, the preorder $L_n^\mu(\theta)$ is in fact linear, since $a \leq^{L_n^\mu(\theta)} b$ if and only if $\bar b = c_j \circ_{n-1} \cdots \circ_{n-1} c_0 = \bar a$ is defined in $\theta$ for some $j$ and sequence of generating cells $c_i$, and the generating cells of $\theta$ are linearly ordered by this composability relation. Moreover, $L_n^\mu(\theta)$ refines $\triangleleft_{n-1}$: since $L_n^\mu(\theta)$ is linear, to see this it suffices to show that if $a \triangleleft_{n-1} b$, then $b \not \leq^{L_n^\mu(\theta)} a$. This follows from acyclicity of $P$. Thus $L_n^\mu$ does indeed map $\Decompat{\{n-1\}}(\mu)$ to $\myPoset_{n}(\mu)$. To see that $L_n^\mu$ is order-preserving, let $\theta' \to \theta$ be an indecomposable morphism in $\Decompat{\{n-1\}}(\mu)$, merging two adjacent generators of $\theta$ into one generator in $\theta'$. Then $L_n^\mu(\theta')$ refines $L_n^\mu(\theta)$ by setting those two equivalence classes to be equal, so that $L_n^\mu(\theta') \leq L_n^\mu(\theta)$.
    
    To see that $L_n^\mu$ is surjective, first observe that \cite[Lem 3.3.1.6]{forest-thesis} shows that every linear order in $\myPoset_n(\mu)$ is in the image of $L_n^\mu$. So it will suffice to show that if $L' \to L$ is an indecomposable morphism in $\lin (\myPoset_n(\mu))$, merging two adjacent equivalence classes $E,F$ of $L$ into one equivalence class $E \cup F$ in $L'$, and if $L = L_n^\mu(\theta)$ is in the image of $L_n^\mu$, then $L'$ is also in the image of $L$. Indeed, by hypothesis $E$ and $F$ correspond to generators in $\theta$ which are composable. Merging these by composition, we obtain $\theta' \in \calC$ with $\theta' \subseteq \theta$. So $L_n^\mu$ is surjective. This argument also shows that $L_n^\mu$ is full. Thus $L_n^\mu$ is an isomorphism of posets.
\end{proof}

\begin{lem}\label{lem:cocart}
    Let $P$ be a torsion-free complex, and $\mu$ be an $n$-cell in $P$. Then the forgetful map 
    \[U^{[0,n-2]}_{[0,n-1]} : \Decompbot(\mu) = \Decompatbot{[0,n-1]}(\mu) \to \Decompatbot{[0,n-2]}(\mu)\]
    is a cocartesian fibration.
\end{lem}
\begin{proof}
    Let $\theta \to \theta'$ be an indecomposable morphism in $\Decompatbot{[0,n-2]}(\mu)$, so that two adjacent generators $a,b$ in $\theta'$ are merged to a single generator $c = b \circ_i a$ in $\theta$. Suppose that $\zeta \in \Decompatbot{[0,n-1]}(\mu)$ has $U^{[0,n-2]}_{[0,n-1]}(\zeta) = \theta$.
    Note that the data of $\zeta$ consists of the data of $\theta$ along with, for each generator $d$ in $\theta$, an element of $\Decompatbot{\{n-1\}}(d)$, which by \cref{prop:pos-iso} amounts to the choice of a linear preorder on $d_n$ refining $\triangleleft_{n-1}$. The map $\theta \to \theta'$ partitions $c_n = a_n \amalg b_n$ in such a way that $a_n, b_n$ are each unions of connected components under $\triangleleft_{n-1}$. A cocartesian lift $\zeta \to \zeta'$ is obtained from $\zeta$ by taking the same linear preorder on $d_n$ for $d \neq a,b$, and restricting the linear preorder given by $\zeta$ on $c_n$ to linear preorders on $a_n$ and $b_n$. For \cref{prop:pos-iso} implies that this is a locally cocartesian lift, and the collection of composites of such lifts is closed under composition, so it is a cocartesian lift.
\end{proof}


\begin{lem}\label{lem:foo}
    Let $P$ be a torsion-free complex with an $n$-cell $\mu$. Let $1 \leq k \leq n$. Suppose that for every $k < i \leq n$, the preorder $\myPoset_i(\mu)$ is an equivalence relation. Then 
    \begin{enumerate}
        \item For fixed $i$, $k < i \leq n$, and distinct $i$-atoms $a,b \in \mu_i$, the $k$-boundaries $\langle a \rangle_k, \langle b \rangle_k$ are disjoint. 
        \item If $k < i \leq j \leq n$ and $a \in \mu_j$, the $k$-boundary $\langle a \rangle_k$ is a union of $k$-boundaries of $i$-atoms $b \in \mu_i$. Namely, it is the union of $\langle b \rangle_k$ for $b \in a_i$.
        \item For $a \in \mu_i$, $k < i \leq n$, we have $\langle a \rangle_k \subseteq \mu_k$.
        \item $\mu$ may be assembled from its atoms using only $\circ_i$ for $i \leq k-1$.
    \end{enumerate}
\end{lem}
\begin{proof}
    (2) holds without any ``equivalence relation" hypothesis. So then too does (1): if $a,b \in \mu_i$ and $x \in \langle a \rangle_k \cap \langle b \rangle_k$, then by induction on $i-k$ we have that $\langle a \rangle_{k+1}$ and $\langle b \rangle_{k+1}$ are disjoint, but by fork-freeness there must be $y$ in their intersection with $x \in y^+$. 

    Note that (4) implies (3) by the composition formula of \cref{def:tf}. And (4) follows by iteratively applying \cite[Lemma 3.3.2.2]{forest-thesis}.
\end{proof}

\begin{lem}\label{lem:bdy-iso}
    Let $P$ be a torsion-free complex with an $n$-cell $\mu$. Then for $k < l \leq m \leq n$, and $S \subseteq [0,k-1]$, if the preorder $\myPoset_j(\partial_m \mu)$ is an equivalence relation for $j \geq l$, then the map 
    \[D_{m,S}^{l,S} : \Decompatbot{S}(\partial_m \mu) \to \Decompatbot{S}(\partial_l \mu)\]
    is fully faithful. Its image comprises those cells $\theta$ such that for every $l < i \leq m$ and every $a \in \mu_i$, there exists a generating $l$-cell $c$ of $\theta$ such that $\langle a \rangle_l \subseteq c_l$.
\end{lem}
\begin{proof} 
    First, if $\theta \in \Decompatbot S (\partial_m \mu)$ and if $a \in \mu_i$ for $l < i \leq m$, then $a \in d_i$ for some generating $i$-cell $d$ of $\theta$. Because the preorder $\myPoset_l(\partial_m \mu)$ is an equivalence relation, by \cref{lem:foo}(3) we have $\langle a \rangle_l \subseteq d_l = (\partial_l d)_l$.
    Moreover, $\partial_l d$ is a generating cell in $D_{m,S}^{l,S}(\theta)$, so the stated condition is satisfied by the image of $D_{m,S}^{l,S}(\theta)$.

    Conversely, if $\zeta \in \Decompatbot{S}(\partial_l \mu)$ satisfies this condition, then consider $b \in \mu_{l+1}$. By hypothesis, $\langle b \rangle_l \subseteq c_l$ for a (necessarily unique) generating $l$-cell $c$ of $\zeta$. In this way we can partition the $(l+1)$-atoms of $\partial_{l+1} \mu$. By fork-freeness of $\mu_{l+1}$, the sets $\langle b \rangle_l$ are also disjoint for the various $b \in \mu_{l+1}$. So for each generating $l$-cell $c$ of $\zeta$, there is an $(l+1)$-cell $c'$ of $\overline{\partial_{l+1} \mu}$ formed by taking $c'_{l+1}$ to be the set $C$ of all $b \in \mu_{l+1}$ with $\langle b \rangle_l \subseteq c_l$, taking $c'_{-l} = C^- \cup c_{-l} \setminus C^+$, and $c'_i = c_i$ for all other $i$. These various $c'$ fit together to form a cell $\theta \in \Decompatbot{S}(\partial_l \mu)$ with $D_{l+1,S}^{l,S}(\theta) = \zeta$. Moreover, if $a \in \mu_i$ for $l+1 < i \leq m$, there is a generating $l$-cell $c$ of $\zeta$ such that $\langle a \rangle_l \subseteq c_l$, and we in fact have $\langle a \rangle_{l+1} \subseteq c'_{l+1}$ because if $x \in \langle a \rangle_{l+1} \setminus c'_{l+1}$, then $\partial_l x \not \in c_l$. Thus by induction on $m-l$, we have that $\theta$ lifts to $\Decompatbot S (\partial_m \mu)$. We conclude that $\Decompatbot S (\partial_l \mu)$ is surjective onto the indicated image.
    
    Moreover, the section we have constructed of $D_{m,S}^{l,S}$ is in fact an inverse because every $\theta \in \Decompatbot S (\partial_m \mu)$ must partition the $(l+1)$-cells as we have done. It is also order-preserving, since if $\zeta_1 \subseteq \zeta_2$ then the cells $c'$ constructed for $\zeta_1$ are just unions of the cells $c'$ constructed for $\zeta_2$, so that $\theta_1 \subseteq \theta_2$. Thus we have an order isomorphism as desired.
\end{proof}

\begin{prop}\label{lem:forest-order}
    Let $P$ be a torsion-free complex with an $n$-cell $\mu$ and let $1 \leq k \leq n$ be such that the preorder $\myPoset_i(\mu)$ is an equivalence relation for $i \geq k+1$.
    Then $L_k^\mu : \Decompat{\{k-1\}}(\mu) \to \lin (\myPoset_k(\mu))$ is an isomorphism of posets.
\end{prop}
\begin{proof}
    From \cref{lem:bdy-iso} and the definition of $\myPoset_{k+1}(\mu)$, we see that the following square is a pullback:
    \begin{equation*}
        \begin{tikzcd}
            \Decompat{\{k-1\}} (\mu) \ar[r,"D_{n,\{k-1\}}^{k,\{k-1\}}"] \ar[d, "L_k^\mu"] & \Decompat{\{k-1\}} (\partial_{k} \mu) \ar[d,"L_k^{\partial_{k} \mu}"] \\
            \lin (\myPoset_{k}(\mu)) \ar[r] & \lin (\myPoset_{k}(\partial_{k} \mu))
        \end{tikzcd}
    \end{equation*}
    By the case $k = n$ (\cref{prop:pos-iso}), the downward arrow on the right is an isomorphism, so the downward arrow on the left is as well.
\end{proof}

\begin{lem}\label{lem:cocart'}
    Let $P$ be a torsion-free complex with an $n$-cell $\mu$ and let $1 \leq k \leq n$ be such that the preorder $\myPoset_i(\mu)$ is an equivalence relation for $i \geq k+1$. Then the functor $U^{[0,k-1]}_{[0,k]} : \Decompatbot{[0,k]}(\mu) \to \Decompatbot{[0,k-1]}(\mu)$ is a cocartesian fibration.
\end{lem}
\begin{proof}
    From \cref{lem:bdy-iso}, we see that the following square is a pullback:
    \begin{equation*}
        \begin{tikzcd}
            \Decompatbot{[0,k]}(\mu) \ar[r,"D_{n,[0,k]}^{k+1,[0,k]}"] \ar[d,"U^{[0,k-1]}_{[0,k]}"] & \Decompatbot{[0,k]}(\partial_{k+1} \mu) \ar[d,"U^{[0,k-1]}_{[0,k]}"] \\
            \Decompatbot{[0,k-1]}(\mu) \ar[r,"D_{n,[0,k-1]}^{k+1,[0,k-1]}"] & \Decompatbot{[0,k-1]}(\partial_{k+1} \mu)
        \end{tikzcd}
    \end{equation*}
    By the case $k = n$ (\cref{lem:cocart}), 
    the downward arrow on the right is a cocartesian fibration, so the downward arrow on the left is as well.
\end{proof}

\begin{lem}\label{lem:cocart''}
    Let $P$ be a torsion-free complex with an $n$-cell $\mu$ and let $1 \leq k \leq n$ be such that the preorder $\myPoset_i(\mu)$ is an equivalence relation for $i \geq k+1$. Then the functor $U^{[0,k-1]}_{[0,n-1]} : \Decompatbot{[0,n-1]}(\mu) \to \Decompatbot{[0,k-1]}(\mu)$ is a cocartesian fibration.
\end{lem}
\begin{proof}
    This follows from \cref{lem:cocart'} because $U^{[0,k-1]}_{[0,n-1]} = U^{[0,k-1]}_{[0,k]} \circ \cdots \circ U^{[0,n-2]}_{[0,n-1]}$ and cocartesian fibrations are stable under composition.
\end{proof}

\subsection{Contractibility of the space of composites}

In this subsection, we prove \cref{thm:sc-main}, showing that the poset $\Decomp(\mu)$ is contractible. The proof involves reducing to showing that $\Decompat{[k,n-1]}(\mu)$ is contractible, and thence to showing that $\Decompat{\{k\}}(\mu)$ is contractible. For this we use the isomorphism and fibration results of this section and the contractibility results of the previous section.

\begin{lem}\label{cor:forest-sd-ideal}
    Let $P$ be a torsion-free complex, and let $\mu \in P$ be an $n$-cell, and let $1 \leq k \leq n$ be such that the preorder $\myPoset_i(\mu)$ is an equivalence relation for $i \geq k+1$ but not an equivalence relation for $i = k$. Then $\Decompat{\{k\}}(\mu)$ is contractible.
\end{lem}
\begin{proof}
    The isomorphism $\Decompat{\{k\}}(\mu) \cong \lin (\myPoset_k(\mu))$ is \cref{lem:forest-order}. The contractibility then follows from \cref{cor:non-disc-contr}.
\end{proof}

\begin{rmk}\label{rmk:eq-sphere}
    If $\myPoset_i(\mu)$ is an equivalence relation for all $i\geq k$, then we still have $\Decompat{\{k\}}(\mu) \cong \lin(\myPoset_k(\mu))$ by \cref{lem:forest-order}, but in this case this poset has the homotopy type of a sphere (cf. \cref{rmk:disc-sphere}).
\end{rmk}

\begin{rmk}\label{rmk:strict-init-cocart}
    Let $U : J \to I$ be a cocartesian fibration of $\infty$-categories, and suppose that $J$ has a strict initial object $\bot$. Let $\mathring J \subset J$ be the full subcategory obtained by deleting all objects isomorphic to $\bot$. Then the inclusion $\mathring J \to J$ is a cocartesian fibration, so that the composite $\mathring J \to J \to I$ is also a cocartesian fibration.
\end{rmk}

\begin{thm}\label{thm:sc-main}
    Let $P$ be a torsion-free complex and $\mu$ an $n$-cell in $P$. Suppose that $\mu$ is composite. Then the poset $\Decomp(\mu)$ is contractible.
\end{thm}
\begin{proof}
    Let $1 \leq k \leq n$ be such that the preorder $\myPoset_i(\mu)$ is an equivalence relation for $i > k$ and is not an equivalence relation for $i = k$; since $\mu$ is composite; $k$ exists by \cref{lem:exists-min}.
    
    We claim that the forgetful map 
    \[ U_{k-2} : \Decomp(\mu) \to \Decompbot(\mu) \xrightarrow{U_{[0,n-1]}^{[0,k-2]}} \Decompatbot{[0,k-2]}(\mu)\]
    is a homotopy equivalence. Since $\Decompatbot{[0,k-2]}(\mu)$ has an initial object, this implies that $\Decomp(\mu)$ is contractible as desired. By \cref{lem:cocart} and \cref{rmk:strict-init-cocart}, $U_{k-2}$ is a cocartesian fibration. So by Quillen’s Theorem A, it will suffice to show that the fibers $U_{k-2}\inv(\theta)$ are contractible for $\theta \in \Decompatbot{[0,k-2]}(\mu)$. If $\theta \neq \mu$, then $U_{k-2}\inv(\theta)$ has an initial object given by $j_{[0,k-2]}^{[0,n-1]} (\theta)$. It remains to verify that $\Decompat{[k-1,n-1]}(\mu) = U_{k-1}\inv(\mu)$ is contractible.

    For this, we apply \cref{lem:special} to the forgetful functor 
    \[ U_{k-1}' : \Decompat{[k-1,n-1]}(\mu) \to \Decompatbot{[k-1,n-1]}(\mu) \xrightarrow{U_{[k-1,n-1]}^{\{k-1\}}} \Decompatbot{\{k-1\}}(\mu)\]
    This functor is the pullback of $U_{k-1}$ to $\Decompatbot{\{k-1\}}(\mu) = \Decompatbot{[k-1,n-1]}(\mu) \cap \Decompatbot{[0,k-1]}(\mu)$, so it is again a cocartesian fibration.
    Clearly $\Decompatbot{\{k-1\}}(\mu)$ has a strict initial object $\mu$. So by \cref{lem:special}, it will suffice to check that $\Decompat{\{k-1\}}(\mu) = \Decompatbot{\{k-1\}}(\mu) \setminus \{\mu\}$ is contractible, and that the fiber $(U_{k-1}')\inv(\theta)$ is contractible for each $\theta \in \Decompat{\{k-1\}}(\mu)$. 
    For the latter, observe that for $\theta \neq \mu$, the fiber $(U_{k-1}')\inv(\theta)$ has an initial object given by $j_{\{k-1\}}^{[k-1,n-1]} (\theta)$. For the former, 
    by \cref{cor:forest-sd-ideal} this poset is contractible since the preorder $\myPoset_k(\mu)$ is not an equivalence relation.
\end{proof}

\begin{rmk}
    Let $P$ be a torsion-free complex and $\mu$ an $n$-cell in $P$. If $\mu$ is not composite (i.e. $\mu$ is an atom), then the poset $\Decomp(\mu)$ is empty. So in general, \cref{thm:sc-main} tells us that the classifying space $|\Decomp(\mu)|$ is a \defterm{proposition} -- it is empty or contractible. The proposition $|\Decomp(\mu)|$ may be read ``$\mu$ is composite".
\end{rmk}

\section{Weak pushouts strictly}\label{sec:ws}
In this section, let $N \in \nats \cup \{\omega\}$. This section is devoted to the proof of Theorems B and C (\cref{thm:mainthm} and \cref{cor:po}). Recall the statement of \cref{thm:mainthm} is that for any  torsion-free complex $P$ of dimension $n \leq N$, the following pushout in $\sCat_N$ is preserved by the inclusion functor $\sCat_N \to \Cat_N^\inc$ (and thence by the the localization to $\Cat_N$):

\begin{equation*}
    \begin{tikzcd}
        P_n \times \partial \globe_n \ar[r] \ar[d] & P_n \times \globe_n \ar[d] \\
        \Free(P_{\leq n-1}) \ar[r] & \Free(P)
    \end{tikzcd}
\end{equation*}

We begin by taking the above pushout in $\Psh(\Theta_N)$ and calling the resulting presheaf $\Free_{n-1} P$ (\cref{def:free}). As $L_{\Cat_N^\inc}$ preserves pushouts, \cref{thm:mainthm} is equivalent to showing that the canonical map $\Free_{n-1} P \to \Free P$ is $L_{\Cat_N^\inc}$-acyclic. In model-category-theoretic terms, $\Free_{n-1} P$ is a model for the pushout in $\Cat_N^\inc$ which is far from fibrant, and we are trying to show that the map $\Free_{n-1} P \to \Free P$ is a fibrant replacement.

To prove this, we factor the inclusion $\Free_{n-1} P \to \Free P$ as a composite of three maps, each of which we will show to be $L_{\Cat_N^\inc}$-acyclic:
\[\Free_{n-1} P \to \Free_\partial P \to \Free^+_\partial P \to \Free P\]

We define $\Free_\partial P$ in \cref{def:free}. It is obtained from $\Free_{n-1} P$ by gluing in all composites of cells which are contained in $\Free P'$ for $P' \subset P$ a proper subcomplex. Using induction on the size of $\Free P$, this map is easily seen to be $L_{\Cat_N^\inc}$-acyclic (\cref{lem:final-piece}). For many torsion-free complexes $P$, we in fact have $\Free_\partial P = \Free P$. The only exception is when $P$ has a (necessarily unique) big cell $\mu$ which is not atomic. In this case the fiber of $(\Free_\partial P)(\theta) \to (\Free P)(\theta)$ at $T : \theta \to \Free P$ is a point except when the big cell $\mu$ is contained in the image $T(\theta)$, in which case the fiber is empty.

To fill in these empty fibers, we define $\Free_\partial^+ P$ in \cref{def:prefib}. Although we give there a ``closed-form formula" for $\Free_\partial^+ P$, it may alternatively be thought of via the skeletal filtration 
\[\Free_\partial P \to \Free_\partial^0 P \to \Free_\partial^1 P \to \cdots \to \Free_\partial^n P = \Free_\partial^+ P\]
appearing in the proof of \cref{lem:fib-rep}. From this perspective, $\Free_\partial^0 P$ is obtained from $\Free_\partial P$ as follows. For every map $T : \theta \to \Free P$ not factoring through $\Free_\partial P$, we have $\theta \cap \Free_\partial P = \Free_\partial \theta$. Assuming that \cref{thm:mainthm} is true for $\theta$ in place of $P$ (which we check in \cref{lem:theta-main} and \cref{cor:theta-main}, using the background from \cref{sec:weak}), we thus have an $L_{\Cat_N^\inc}$-acyclic extension $\Free_\partial P \to \Free_\partial P \cup_{\Free_\partial \theta} \theta$. Let $[T]$ denote the inclusion $\theta \to \Free_\partial P \cup_{\Free_\partial \theta} \theta$. Taking the union of these pushouts over all such $\mu \in \theta \xrightarrow T \Free P$ gives us $\Free_\partial^0 P$.

At this point, the fibers of $\Free_\partial^0 P \to \Free P$ are now all nonempty and discrete, but not connected. For example, if there are two different cells $\zeta, \theta \to \Free P$ containing $\mu$, then pulling back along the inclusion of the ``big cells" $\globe_n \to \zeta$, $\globe_n \to \theta$ produces two distinct elements of $(\Free_\partial^+ P)(\globe_n)$ which both map to $\mu$ in $(\Free P)(\globe_n)$. We correct this ``over-production" inserting edges between such cells as follows. Whenever we have $\mu \in \theta_0 \subset \theta_1 \xrightarrow{T} \Free P$, we have two distinct elements of $(\Free_\partial^0 P)(\theta_0)$, namely $[T|_{\theta_0}]$ and $[T]|_{\theta_0}$, both of which map to $T|_{\theta_0} \in (\Free P)(\theta_0)$. This gives us two maps $\theta_0 \rightrightarrows \Free_\partial^0 P$, and the restrictions of these two maps to $\Free_\partial \theta_0$ factor through $\Free_\partial P$ and are equal. Thus we have a map $\Free_\partial \theta_0 \times \Delta[1] \cup_{\Free_\partial \theta_0 \times \partial \Delta[1]} \theta_0 \times \partial \Delta[1] \to \Free_\partial^0 P$. Pushing out along the $L_{\Cat_N^\inc}$-acyclic map $\Free_\partial \theta_0 \times \Delta[1] \cup_{\Free_\partial \theta_0 \times \partial \Delta[1]} \theta_0 \times \partial \Delta[1] \to \theta_0 \times \Delta[1]$ gives us $\Free_\partial^1 P$. That is, we have glued in paths between pairs of ``redundant" cells in $\Free_\partial^0 P$. 

The fibers of $\Free_\partial^1 P \to \Free P$ are now connected (because the composite $\mu$ is unique in $\Free P$ by the usual theory of torsion-free complexes), but not simply-connected. Continuing in this manner involves pushout-products of $\Free_\partial \theta_0 \to \theta_0$ against $\partial\Delta[k] \to \Delta[k]$ for larger and larger $k$ and corresponding to chains $\theta_0 \subset \cdots \subset \theta_k \subseteq \Free P$, and eventually results in a $\Theta$-space called $\Free_\partial^+ P$. At the end of the construction, we see that the $k$-simplices  of $(\Free_\partial^+ P)$ not contained in $(\Free_\partial P)(\zeta)$ correspond to chains $\mu \in \zeta \to \theta_0 \subseteq \theta_1 \subseteq \cdots \subseteq \theta_k \subseteq \Free P$ such that $\mu \neq \theta_0$. This is precisely the nerve of the slice category indicated in \cref{def:prefib}.

From this perspective, the map $\Free_\partial P \to \Free_\partial^+ P$ is $L_{\Cat_N^\inc}$-acyclic by construction. Finally, the third map $\Free_\partial^+ P \to \Free P$ is not only $L_{\Cat_N^\inc}$-acyclic, but in fact a levelwise equivalence (i.e. an equivalence of $\Theta_N$-spaces) (\cref{prop:level}). The observation is that the only nontrivial fiber of the map is precisely the nerve of the poset $\Decomp(\mu)$ from \cref{sec:cs}, which we showed to be contractible in \cref{thm:sc-main}.

We complement \cref{thm:mainthm} with \cref{cor:po}, which asserts that certain pushouts are preserved by the canonical functor $\sCat_N \to \Cat_N^\inc$ (and hence also by the composite $\sCat_N \to \Cat_N^\inc \to \Cat_N$).

\begin{Def}\label{def:free}
    Let $P$ be a torsion-free complex of dimension $n$. We denote by $\Free_{n-1} P = \Free (P_{\leq n-1}) \cup P_n \in \Psh_\Set(\Theta_N)$ (the union is taken in $\Psh_\Set(\Theta_N)$; see \cref{sec:weak} for conventions regarding presheaves on $\Theta_N$).
\end{Def}

Recall that \cref{thm:mainthm} will be the statement that the inclusion $\Free_{n-1} P \to \Free P$ is $L_{\Cat_N^\inc}$-acyclic. We begin by proving the special case of \cref{thm:mainthm} where the torsion-free complex $P$ (or rather the free strict $N$-category $\Free P$) is assumed to lie in $\Theta$:

\begin{prop}\label{lem:theta-main}
Let $\theta \in \Theta$ and let $T$ be the torsion-free complex with $\Free T = \theta$. Then the canonical map $\Free_{n-1} T \to \Free T$ is $L_{\Cat_N^\inc}$-acyclic. In other words, the following pushout square in $\sCat_N$:
    \begin{equation*}
        \begin{tikzcd}
            T_n \times \partial \globe_n \ar[r] \ar[d] & T_n \times \globe_n \ar[d] \\
            \Free(T_{\leq n-1}) \ar[r] & \Free(T)
        \end{tikzcd}
    \end{equation*}
is a homotopy pushout.
\end{prop}
\begin{proof}
The final statement is equivalent to the first statement because the following pushout diagram in $\Psh(\Theta_N)$:
    \begin{equation*}
        \begin{tikzcd}
            T_n \times \partial \globe_n \ar[r] \ar[d] & T_n \times \globe_n \ar[d] \\
            \Free(T_{\leq n-1}) \ar[r] & \Free_{n-1}(T)
        \end{tikzcd}
    \end{equation*}
is preserved by the localization $L_{\Cat_N^\inc}$ (and thence by the localization $L_{\Cat_N}$).

We induct on the structure of $\theta \in \Theta$. If $\theta \in \Theta_0$, the result is trivial. If there is a nontrivial wedge decomposition $\theta = \Free T_1 \vee \Free T_2$ for $\Free T_1,\Free T_2 \in \Theta$, then by induction $\Free T_i = \Free_{n-1} T_i \ast_{(T_i)_n \times \partial \globe_n} (T_i)_n \times \globe_n$ is a homotopy pushout. By commuting this homotopy pushout with the one from the wedge sum (\cref{lem:wedge}), it results that $\Free T = \Free_{n-1} T \ast_{T_n \times \partial \globe_n} \to T_n \times \globe_n$ is a homotopy pushout as desired. Otherwise, $\theta = \Sigma \Free T'$ for $\Free T' \in \Theta$. By induction, the pushout $\Free T' = \Free_{n-2} T' \ast_{(T')_{n-1} \times \partial \globe_{n-1}} (T')_{n-1} \times \globe_n$ is a homotopy pushout. As homotopy pushouts are stable under suspension (\cref{obs:hty-susp}), the lemma results.
\end{proof}

We now construct our first, ``more-fibrant" version of $\Free_{n-1} P$, called $\Free_\partial P$. 

\begin{Def}\label{def:free-partial}
    We denote by $\Free_\partial P = \Free_{n-1} P \cup \cup_{P' \subsetneqq P} \Free P' \in \Psh_\Set(\Theta_N)$ (again the union is taken in $\Psh_\Set(\Theta_N)$). Note that $\Free_{n-1} P \subseteq \Free_\partial P \subseteq \Free P$.
\end{Def}

As discussed above, $\Free_\partial P$ adjoins to $\Free_\partial P$ all $n$-cells which are contained in $\Free P'$ for a proper subcomplex $P' \subset P$, and all $\theta$-cells whose composite is such a cell. As mentioned in the introduction, if $P$ does not have a big cell $\mu$, or if its big cell $\mu$ is an atom, then $\Free_\partial P = P$.  More generally, the fiber of $\Free_\partial P \to \Free P$ at $\theta \subseteq \Free P$ is a point unless the composite of $\theta$ is the big cell $\mu$ of $\Free P$, in which case it is empty.

The fact that $\Free_{n-1} P \to \Free_\partial P$ is an equivalence uses induction:

\begin{lem}\label{lem:final-piece}
    Let $P$ be a torsion-free complex of dimension $n$. Assume (as we shall for induction in \cref{thm:mainthm} below) that for each proper subcomplex $P' \subset P$, the inclusion $\Free_{n-1} P' \to \Free P'$ is $L_{\Cat_N^\inc}$-acyclic. Then the map $\Free_{n-1} P \to \Free_\partial P$ is $L_{\Cat_N^\inc}$-acyclic.
\end{lem}
\begin{proof}
    Note that if $P', P'' \subseteq P$ are subcomplexes, then $P' \cap P''$ is a subcomplex and $\Free (P' \cap P'') = \Free (P') \cap \Free (P'')$. So by \cref{lem:dist-lat}, it will suffice to show that for each $P' \subset P$, the inclusion $\Free_{n-1} P \to \Free_{n-1} P \cup \Free P'$ is $L_{\Cat_N^\inc}$-acyclic. And indeed, this map is a cobase-change along a monomorphism of $\Free_{n-1} P' \to \Free P'$, which is $L_{\Cat_N^\inc}$-acyclic by hypothesis.
\end{proof}

We may already deduce a variation of the main theorem in the case where $\Free P \in \Theta$, which will be essential to the next step:

\begin{cor}\label{cor:theta-main}
    Let $\theta \in \Theta$ be of dimension $n$ and let $T$ be the torsion-free complex with $\Free T = \theta$. Assume (as we shall for induction in \cref{thm:mainthm} below) that for each proper subcomplex $T' \subset T$, the inclusion $\Free_{n-1} T' \to \Free T'$ is $L_{\Cat_N^\inc}$-acyclic. Then the map $\Free_\partial T \to \Free T$ is $L_{\Cat_N^\inc}$-acyclic.
\end{cor}
\begin{proof}
    Consider the two maps $\Free_{n-1} T \to \Free_\partial T \to \Free T$. By \cref{lem:final-piece} the first is $L_{\Cat_N^\inc}$-acyclic, and by \cref{lem:theta-main} the composite is $L_{\Cat_N^\inc}$-acyclic. The result follows by two-for-three.
\end{proof}

We next build a canonical ``pre-fibrant replacement" $\Free^+_\partial P$ of $\Free_\partial P$ in $\Psh_{\sSet}(\Theta_N)$. It is a ``replacement" in the sense that $\Free_\partial P \to \Free^+_\partial P$ is $L_{\Cat_N^\inc}$-acyclic (\cref{lem:fib-rep}), and it is ``pre-fibrant" in the sense that 
it admits many more lifts along spine inclusions than $\Free_\partial P$ itself does. Indeed, it will fill in all the missing $\theta$-cells with composite the big cell $\mu$ many times over, and then correct the over-generation of cells with higher paths between them, and so forth.

\begin{Def}\label{def:prefib}
    Let $P$ be a torsion-free complex of dimension $n$. We let $\Free^+ P$ denote the $\Cat$-valued presheaf on $\Theta_N$ given by $\Free^+ P (\zeta) = \zeta \downarrow (\Theta_N \setminus \{\globe_n\}) \downarrow_\nd \Free P \setminus \zeta \downarrow \mu$, which we regard as a $\sSet$-valued presheaf by taking its nerve levelwise. Here, an object of $\Free^+ P(\zeta)$ consists of maps $\zeta \to \theta \to \Free P$ where $\theta \in \Theta_N$, $\theta \neq \globe_n$, and the second map $\theta \to \Free P$ is nondegenerate (hence monic, by $\Theta$-regularity of $\Free P$). A morphism $(\zeta \to \theta \to \Free P) \to (\zeta \to \theta' \to \Free P)$ is a map $\theta \to \theta'$, necessarily monic, making two triangles commute. There is a canonical ``collapse" map $\gamma : \Free^+ P \to \Free P$, which carries $\zeta \to \theta \to \Free P$ to the composite $\zeta \to \Free P$.

    We let $\Free^+_\partial P = \Free^+ P \cup_{\Free_{n-1} P \cup_{P' \subsetneq P} \Free^+ P'} \Free_\partial P$, where the pushout is taken in $\Psh_\sSet(\Theta_N)$. There is an inclusion $\Free_\partial P \to \Free^+_\partial P$, as well as an induced collapse map $\gamma_\partial : \Free^+_\partial P \to \Free P$. The composite is the canonical inclusion $\Free_\partial P \to \Free P$.
\end{Def}

Note that the pushout constructing $\Free^+_\partial P$ collapses to a point any simplex whose image in $\Free P$ is contained in $\Free_\partial P$. Thus the fiber of $(\Free_\partial^+ P)(\theta)$ over $T : \theta \to \Free P$ is a point if $T(\theta) \subseteq \Free_\partial P$. Otherwise, $\mu$ is in the image of $\theta$, and the fiber is the subcategory of $\theta \downarrow (\Theta \setminus \{\globe_n\}) \downarrow_\nd \Free P$ where the composite $\theta \to \Free P$ is required to be $T$. If $T$ is nondegenerate, then this subcategory has an initial object given by $T$ itself, except in the case where $\theta = \globe_n$ and $T = \mu$. In that case, we have the category $(\mu \downarrow \Theta \downarrow_\nd \Free P) \setminus \{\mu\}$, which is precisely the category $\Decomp(\mu)$ from \cref{def:decomp}. Thus we can see that $\Free_\partial^+ P \to \Free P$ is an equivalence:

\begin{prop}\label{prop:level}
    Let $P$ be a torsion-free complex of dimension $n$. Then the collapse map $\gamma_\partial : \Free^+_\partial P \to \Free P$ is a levelwise equivalence (i.e. an equivalence in $\Psh(\Theta_N)$).
\end{prop}
\begin{proof}
    If $P$ does not have a big cell, or has a big cell which is atomic, then $\gamma_\partial$ is an isomorphism. Otherwise, consider the fiber of $\gamma_\partial\inv(T : \theta \to \Free P)$ (note that fibers are the same as homotopy fibers since $\Free P(\theta)$ is discrete). We wish to show that $\gamma_\partial\inv(T)$ is contractible. It suffices to consider the case where $T$ is nondegenerate, and hence injective by hypercancellativity. If $T \subseteq \Free P'$ for some $P' \subsetneq P$, $\gamma_\partial\inv(T)$ is in fact a single point. Otherwise, we have $\mu \in T(\theta)$. If $\mu \subsetneq T(\theta)$, then $\gamma_\partial\inv(T) = T \downarrow (\Theta  \setminus \{\globe_n\}) \downarrow_\nd \Free P$ is a poset with an initial object given by $\theta = \theta \xrightarrow T \Free P$. Finally, if $T$ is the inclusion $\mu : \globe_n \to \Free P$, then $\gamma_\partial\inv(\mu)$ is the poset $\Decomp(\mu) = (\mu \downarrow \Theta \downarrow_\nd \Free P) \setminus \{\mu\}$ from \cref{sec:cs}. This was shown to be contractible in \cref{thm:sc-main}.
\end{proof}

It remains to show that $\Free_\partial P \to \Free_\partial^+ P$ is $L_{\Cat_N^\inc}$-acyclic. Our approach to making precise the outline given at the beginning of the section is to examine the skeletal filtration, and show that each step is $L_{\Cat_N^\inc}$-acyclic using \cref{cor:theta-main}:

\begin{lem}\label{lem:fib-rep}
Let $P$ be a torsion-free complex of dimension $n$. 
Then the inclusion $\Free_\partial P \to \Free^+_\partial P$ is $L_{\Cat_N^\inc}$-acyclic.
\end{lem}
\begin{proof}
If $P$ does not have a big cell $\mu$, then the map is an isomorphism. Otherwise, consider the skeletal filtration:
\begin{equation*}
    \begin{tikzcd}
        \Free_\partial P \ar[r,phantom,"=:" description] \ar[drr] &
        \Free^{-1}_\partial P \subseteq \Free^0_\partial P \subseteq \ar[r,phantom,"\cdots" description]  &
        \subseteq \Free^n_\partial P = \Free^+_\partial P \ar[d] \\
        & & \Free P
    \end{tikzcd}
\end{equation*}
Here $(\Free^k_\partial P)(\theta) \subseteq (\Free^+_\partial P)(\theta)$ is the union of $(\Free_\partial P)(\theta)$ and the $k$-skeleton of $(\Free^+_\partial P)(\theta)$. We claim that for each $k$, the inclusion $\Free^{k-1}_\partial P \to \Free^k_\partial P$ is $L_{\Catinc}$-acyclic. Indeed, let $\sigma \in (\Free^k_\partial P)(\theta) \setminus (\Free^{k-1}_\partial P)(\theta)$ be nondegenerate in both the $\Theta$ and $\Delta$ directions. Note that we must have $\mu \in \theta$. The map $\sigma$ is classified by a map $s : \theta \times \Delta[k] \to \Free^k_\partial P$, which corresponds to a chain:
\begin{equation}\label{eqn:chain}
    \mu \subseteq \theta = \theta_{-1} \subseteq \theta_0 \subsetneq \cdots \subsetneq \theta_{k} \subseteq \Free P
\end{equation}
(The map $\theta_{-1} \to \theta_0$ is monic because $\sigma$ is nondegenerate in the $\Theta$ direction, the subsequent maps are monic by definition of $(-)^+$, and the subsequent maps are not identities because $\sigma$ is nondegenerate in the $\Delta$-direction.) We claim that 
\begin{equation}\label{eqn:hi}
    s\inv(\Free^{k-1}_\partial P) = \theta \times \partial \Delta[k] \cup_{\Free_\partial T \times \partial \Delta[k]} \Free_\partial T \times \Delta[k] \subset \theta \times \Delta[k]
\end{equation}
where $T$ is the torsion-free complex presenting $\theta$. By \cref{cor:theta-main} and cartesianness of the model structure, this map is $L_{\Catinc}$-acyclic. We claim also that $s$ is injective on the complement of $s\inv(\Free^{k-1}_\partial P)$. It follows that the extension $\Free^{k-1}_\partial \subset \Free^k_\partial P$ is obtained by pushout along the above inclusion for all such $s$ where $\theta_{-1} \to \theta_0$ is the identity, and the lemma results.

To see that \cref{eqn:hi} holds, first observe that $\theta \times \partial \Delta[k] \subseteq \Free^{k-1}_\partial P$ by definition of the skeletal filtration. Moreover, the image of $\Free_\partial T \times \Delta[k]$ in $\Free P$ is contained already in $\Free_\partial P$ (because each inert subobject of $\theta$ has image contained in some $\Free P' \subsetneq \Free P$), so that the image in $\Free_\partial^+ P$ is contained in $\Free_\partial P$, and in particular the image in $\Free^k_\partial P$ is contained in $\Free^{k-1}_\partial P$. Thus we have $s\inv(\Free^{k-1}_\partial P) \supseteq (\theta \times \partial \Delta[k]) \cup (\Free_\partial T \times \Delta[k])$. Conversely, if $Z \in (\theta \times \Delta[k])(\zeta)_k \setminus (\Free_\partial T \times \Delta[k])(\zeta)_k $ is a nondegenerate $k$-simplex, then $\zeta \to \theta$ must be active, so the chain in \cref{eqn:chain} remains nondegenerate and outside of $(\Free_\partial P)(\zeta)$ upon restriction to $\zeta$. Thus $s(Z)$ is not in $(\Free^{k-1}_\partial P)(\zeta)$. This argument also shows that $s$ is injective on such simplices, since there is no room for such chains to be identified with one another when $\zeta \to \theta$ is active.
\end{proof}

\begin{thm}\label{thm:mainthm}
    Let $P$ be a torsion-free complex of dimension $n \leq N \in \nats \cup \{\omega\}$. Then the map $\Free_{n-1} P \to \Free P$ is $L_{\Cat_N^\inc}$-acyclic (and hence also $L_{\Cat_N}$-acyclic).
    In other words, the following pushout in $\sCat_N$:
    \begin{equation*}
        \begin{tikzcd}
            P_n \times \partial \globe_n \ar[r] \ar[d] & P_n \times \globe_n \ar[d] \\
            \Free(P_{\leq n-1}) \ar[r] & \Free(P)
        \end{tikzcd}
    \end{equation*}
    is preserved by the inclusion $\sCat_N \to \Cat_N^\inc$ (and thence by the composite $\Cat_N \to \Cat_N^\inc \to \Cat_N$).
\end{thm}
\begin{proof}
    Because weak equivalences are stable under filtered colimits, it suffices to treat the case where $P$ is finite. We induct on the dimension $n$ of $P$ and on the number of cells in $P$. When $P$ is empty, the result is trivial.
    The map $\Free_{n-1} P \to \Free P$ factors as
    \begin{equation*}
        \Free_{n-1} P \to \Free_\partial P \to \Free^+_\partial P \to \Free P
    \end{equation*}
    The first map is $L_{\Catinc}$-acyclic by \cref{lem:final-piece}. The second map is $L_{\Catinc}$-acyclic by \cref{lem:fib-rep}. The third map is an equivalence by \cref{prop:level}.

    The final statement follows because the following pushout diagram in $\Psh(\Theta_N)$:
    \begin{equation*}
        \begin{tikzcd}
            P_n \times \partial \globe_n \ar[r] \ar[d] & P_n \times \globe_n \ar[d] \\
            \Free(P_{\leq n-1}) \ar[r] & \Free_{n-1}(P)
        \end{tikzcd}
    \end{equation*}
    is preserved by the localization $L_{\Cat_N^\inc}$ (and thence by the localization $L_{\Cat_N}$).
\end{proof}

\begin{cor}\label{cor:po}
Let $n \leq N \in \nats \cup \{\omega\}$. Let
\begin{equation*}
    \begin{tikzcd}
        A \ar[r,tail] \ar[d,hook] \ar[dr,phantom, "\ulcorner"{description,very near end}] &
        B \ar[d, hook] \\
        C \ar[r,tail] & D
    \end{tikzcd}
\end{equation*}
be a pushout diagram of strict $n$-categories where the downward maps are folk cofibrations and the rightward maps are monic. Suppose that $A,B,C,D$ are free on torsion-free complexes. Then this pushout is preserved by the inclusion $\sCat_N \to \Cat_N^\inc $ (and hence also by the composite $\sCat_N \to \Cat_N^\inc \to \Cat_N$).
\end{cor}
\begin{proof}
It suffices to consider the case when $A \hookrightarrow C$ is the inclusion $\partial \globe_d \hookrightarrow \globe_d$. We may assume that the cell being attached was the last cell to be glued on to build the computad $D$. The result follows from \cref{thm:mainthm} by cancelling pushout squares.
\end{proof}

\begin{rmk}
    The generality of torsion-free complexes is convenient, since for example (\cite[Thm 3.4.4.22]{forest-thesis}, cf. \cref{rmk:general}) they include all loop-free augmented directed complexes in the sense of Steiner \cite{steiner}. In particular, \cref{thm:mainthm} and \cref{cor:po} apply when the categories involved are objects of Joyal's category $\Theta$, when they are orientals in the sense of Street, when they are lax Gray cubes, and when they are subcomplexes of any of these. Compare \cite{maehara-orientals}.
\end{rmk}

\bibliographystyle{alpha}
\bibliography{pasting}
\end{document}